\documentclass[reqno,oneside,11pt]{amsart}
\usepackage[T1]{fontenc}
\usepackage[utf8]{inputenc}
\usepackage{lmodern}
\usepackage[french,english]{babel}
\usepackage{geometry} 

\usepackage{amsmath,amssymb,amsfonts,amsthm,mathtools}
\usepackage{mathrsfs}
\usepackage{graphicx}
\usepackage{caption}
\usepackage{subcaption} 
\usepackage{booktabs}
\usepackage{array}
\usepackage{paralist}
\usepackage{fancyhdr}
\usepackage{emptypage}
\usepackage{comment} 
\usepackage{xcolor}
\usepackage{epigraph}

\usepackage[noadjust]{cite}
\usepackage[
pagebackref=true,
colorlinks=true,
urlcolor=purple,
linkcolor=purple!87!black,
citecolor=green!60!black,
pdfborder={0 0 0}
]{hyperref}
\renewcommand*{\backref}[1]{}
\renewcommand*{\backrefalt}[4]{[{\tiny%
		\ifcase #1 Not cited.%
		\or Cited on page~#2.%
		\else Cited on pages #2.%
		\fi%
	}]}
\usepackage{bookmark}
\usepackage{dsfont}

\newcommand{\cay}[2]{\mathrm{Cay}(#1,#2)}

\newcommand{\Z}{\mathbb{Z}}
\newcommand{\R}{\mathbb{R}}
\renewcommand{\P}{\mathbb{P}}
\newcommand{\E}{\mathbb{E}}
\newcommand{\supp}[1]{\mathrm{supp}(#1)}

\newcommand{\thistheoremname}{}

\newtheorem*{genericthm*}{\thistheoremname}
\newenvironment{namedthm*}[1]
{\renewcommand{\thistheoremname}{#1}%
	\begin{genericthm*}}
	{\end{genericthm*}}
\theoremstyle{plain}
\newtheorem{thm}{Theorem}[section] 
\newtheorem{prop}[thm]{Proposition}
\newtheorem{lem}[thm]{Lemma}
\newtheorem{cor}[thm]{Corollary}
\theoremstyle{definition}
\newtheorem{rem}[thm]{Remark}
\newtheorem{defn}[thm]{Definition} 
\newtheorem{exmp}[thm]{Example} 
\newtheorem{question}[thm]{Question}

\usepackage{cleveref}

\newcommand{\prange}[2]{(\labelcref{#1})--(\labelcref{#2})}

\author{Joshua Frisch}
\email[Joshua Frisch]{joshfrisch@gmail.com}
\address{Department of Mathematics, UC San Diego, 9500 Giman Dr., La Jolla, CA 92093, USA.}
\author{Eduardo Silva} 
\email[Eduardo Silva]{eduardo.silva@uni-muenster.de, edosilvamuller@gmail.com}
\address{University of Münster, Einsteinstrasse 62, Münster 48149, Germany}
\urladdr{\url{https://edoasd.github.io/eduardo_silva_math/}}
\title[The Poisson boundary of wreath products]{The Poisson boundary of wreath products}
\begin{document}
	
	\maketitle
	\begin{abstract} 
		We give a complete description of the Poisson boundary of wreath products $A\wr B\coloneqq \bigoplus_{B} A\rtimes B$  of countable groups $A$ and $B$, for probability measures $\mu$ with finite entropy where lamp configurations stabilize almost surely. If, in addition, the projection of $\mu$ to $B$ is Liouville, we prove that the Poisson boundary of $(A\wr B,\mu)$ is equal to the space of limit lamp configurations, endowed with the corresponding hitting measure. In particular, this answers an open question asked by Kaimanovich \cite{Kaimanovich2001} and Lyons-Peres \cite{LyonsPeres2021} for $B=\Z^d$, $d\ge 3$, and measures $\mu$ with a finite first moment.
	\end{abstract}
\section{Introduction}\label{introduction}
The classical \emph{Poisson integral representation formula} (see e.g.\ \cite[Theorems 4.22 and 4.23]{Ahlfors1978}) establishes an isometric isomorphism between the Banach space of bounded harmonic functions on the disk $\mathbb{D}$ and the space of bounded measurable functions on the circle $\partial \mathbb{D}$. It is possible to find an analogous duality for countable groups as follows. Let $G$ be a countable group and $\mu$ a probability measure on $G$. A function $f:G\to \R$ is said to be $\mu$-harmonic if $f(g)=\sum_{h\in G}f(gh)\mu(h)$ for every $g\in G$. The \emph{Poisson boundary} of $(G,\mu)$ is a measure space endowed with a measurable $G$-action with the following property: there is an isometric isomorphism between the space of bounded $\mu$-harmonic functions on $G$ and the space of bounded measurable functions on the Poisson boundary (see \cite[Theorem 3.1]{Furstenberg1971} and the introductions of \cite{KaimanovcihVershik1983,Kaimanovich2000}). Equivalently, the Poisson boundary encodes the asymptotic behavior of the trajectories of the \emph{$\mu$-random walk} $\{w_n\}_{n\ge 0}$ on $G$, which is the Markov chain with state space $G$ and with independent increments $\{w_{n-1}^{-1}w_n\}_{n\ge 1}$ distributed according to $\mu$. 

Let us denote by $(G^{\mathbb{N}},\P)$ the space of sample paths of the $\mu$-random walk on $G$.

\begin{defn}\label{defn: Poisson boundary original def} Let $G$ be a countable group, and let $\mu$ be a probability measure on $G$. Two sample paths $\mathbf{w}=(w_1,w_2,\ldots)$, $\mathbf{w^\prime}=(w^\prime_1,w^\prime_2,\ldots)\in G^{\mathbb{N}}$ are said to be equivalent if there exist $p,N\ge 0$ such that $w_n=w^\prime_{n+p}$ for all $n>N$. Consider the measurable hull associated with this equivalence relation. That is, the $\sigma$-algebra formed by all measurable subsets of the space of trajectories $(G^{\mathbb{N}},\P)$ which are unions of the equivalence classes of $\sim$ up to $\P$-null sets. The associated quotient space is called the \emph{Poisson boundary} of the random walk $(G,\mu)$.
\end{defn}

There are several other equivalent definitions of the Poisson boundary, and we mention some of them in Section \ref{section: preliminaries}.

The measure $\mu$ is called \emph{adapted} (resp.\ \emph{non-degenerate}) if the support of $\mu$ generates $G$ as a group (resp.\ as a semigroup). Another name for adapted measures in the literature is \emph{irreducible}. We say that an adapted probability measure $\mu$ on a group $G$ has the \emph{Liouville property} if every bounded $\mu$-harmonic function is constant. This is equivalent to $(G,\mu)$ having a trivial Poisson boundary. A natural question to ask is whether a random walk $(G,\mu)$ has the Liouville property. If the random walk does not have this property, the next problem is to realize the Poisson boundary as an explicit measure space related to the geometry of $G$. These problems have been studied in the last decades for various families of groups under different conditions on the measure $\mu$. We refer to Subsection \ref{subsection: background} for a more detailed description of some of these results and the corresponding references.

A key quantity in the study of Poisson boundaries is \emph{entropy}. Given a probability measure $\mu$ on $G$, its entropy is $H(\mu)\coloneqq -\sum_{g\in G}\mu(g)\log(\mu(g))$. 

\begin{defn}\label{defn: asymptotic entropy} The entropy of the random walk $(G,\mu)$, also called the \emph{asymptotic entropy}, is defined as
	\(h_{\mu}\coloneqq\lim_{n\to \infty} \frac{H(\mu^{*n})}{n}.\)
\end{defn}
This quantity was introduced by Avez \cite{Avez1972}, who proved that if $\mu$ is finitely supported and $h_\mu=0$ then $(G,\mu)$ has a trivial Poisson boundary. Furthermore, for any probability measure with $H(\mu)<\infty$, the \emph{entropy criterion} of Derriennic \cite{Derrienic1980} and Kaimanovich-Vershik \cite{KaimanovcihVershik1983} states that the triviality of the Poisson boundary is equivalent to the vanishing of $h_{\mu}$. This criterion was extended to a \emph{conditional entropy criterion} by Kaimanovich \cite[Theorem 2]{Kaimanovich1985conditional}, \cite[Theorem 4.6]{Kaimanovich2000}, that establishes whether a candidate for the Poisson boundary is equal to it. This result is often applied via the associated \emph{ray criterion} \cite[Theorem 5.5]{Kaimanovich2000} and \emph{strip criterion} \cite[Theorem 6.4]{Kaimanovich2000}, which are deduced from it.

It is known that any adapted random walk on an abelian group has a trivial Poisson boundary \cite{Blackwell1955} (see also \cite{ChoquetDeny1960,DoobSnellWilliamson1960}). The same holds for nilpotent groups, which was first proved for non-degenerate probability measures in \cite{DynkinMaljutov1961} and for general probability measures in \cite[Théorème IV.1]{Azencott1970} (see also \cite[Lemma 4.7]{Jaworski2004}). For non-degenerate measures, it is furthermore proved in \cite{Margulis1966} that every positive harmonic function on a nilpotent group is constant. Generalizing the case of nilpotent groups, the Poisson boundary is trivial for any adapted measure on a countable hyper-FC-central group \cite{LinZaidenberg1998,Jaworski2004}, \cite[Proposition 5.10]{ErschlerKaimanovich2023}. These results are satisfied by all adapted step distributions without requiring any additional hypothesis. In contrast, results describing a \emph{non-trivial} Poisson boundary of a random walk on a group usually assume some control over the tail decay of the measure $\mu$, commonly through the finiteness of some moment. The only known results that describe a non-trivial Poisson boundary for all probability measures $\mu$ with finite entropy, without assuming extra conditions on $\mu$, are \cite[Theorem 1.2]{ForghaniTiozzo2019} for free non-abelian semigroups and \cite{ChawlaForghaniFrischTiozzo2022} for hyperbolic groups and more generally acylindrically hyperbolic groups. 

 There are families of groups that possess a geometric boundary that, for any adapted random walk, can be endowed with a probability measure so that the associated probability space is an equivariant quotient of the Poisson boundary. Such a quotient is called a \emph{$\mu$-boundary}, see Subsection \ref{subsection: rw and poisson boundary}. This is the case for any adapted random walk on a free group $F_k$, for which the geometric boundary is the space $\partial F_k$ of infinite reduced words \cite[Section 4]{Furstenberg1967}, and more generally, for hyperbolic groups with the Gromov boundary (see \cite[Corollary 1]{Woess1993} and \cite[Theorem 7.6]{Kaimanovich2000}). Hence, in these cases it is a natural problem to ask whether the geometric boundary corresponds to the Poisson boundary, without any extra assumptions on the measure $\mu$. This is currently an open problem, and there are no results in this direction for adapted probability measures with infinite entropy. The situation is different for amenable groups. Amenability is equivalent to the existence of an adapted probability measure with a trivial Poisson boundary, and thus it is not possible to find a geometric boundary that describes the Poisson boundary for all adapted measures. An important family of groups covered by the results of the current paper are wreath products $A\wr \Z^d$, $d\ge 1$, for $A$ a countable group. If $A$ is amenable, then $A\wr \Z^d$ is also amenable. Our results provide a complete description of the Poisson boundary of random walks on $A\wr \Z^d$, whenever the step distribution of the random walk has finite entropy and where the natural candidate for the Poisson boundary of a wreath product is well-defined (see Theorem \ref{thm: lamplighter Poisson boundary} and Corollary \ref{cor: main cor of main theorem} below). 

We also remark that there are almost no results that identify a non-trivial Poisson boundary for infinite entropy measures on groups. One exception that we are aware of is \cite[Theorem A]{ErschlerKaimanovich2023}, which describes the Poisson boundary of countable ICC groups for measures from the construction of \cite{FrischHartmanTamuzVahidi2019}, that can be chosen to have infinite entropy. We mention that \cite{ForghaniKaimanovichINPREP} describes the Poisson boundary of the free semigroup for measures with a finite first logarithmic moment without assuming finite entropy (see also \cite[Theorem 1.2]{ForghaniTiozzo2019} and \cite[Theorem 3.6.3]{ForghaniThesis2015}).

\subsection{Main results}\label{subsection: main results}
Given $A$ and $B$ countable groups, their \emph{wreath product} $A\wr B$ is the semi-direct product $\bigoplus_{B}A\rtimes B$, where $B$ acts on the direct sum by translations (we remark that some authors use the notation $B\wr A$). Any random walk $\{w_n\}_{n\ge 0}$ on $A\wr B$ can be decomposed using the semi-direct product structure as $w_n=(\varphi_n,X_n)$, $n\ge 0$. Here $\varphi_n\in \bigoplus_{B}A$ is called the \emph{lamp configuration at instant $n$}, and in general it is not a random walk on $\bigoplus_{B}A$. On the other hand, $X_n$ is called the \emph{base position at instant $n$}, and it does describe a random walk on the \emph{base group} $B$. 

Kaimanovich and Vershik \cite[Section 6]{KaimanovcihVershik1983} studied the groups $\Z/2\Z\wr \Z^d$, $d\ge 1$, and showed that if the step distribution is finitely supported and the random walk $\{X_n\}_{n\ge 0}$ on $\Z^d$ is transient, then the lamp configurations $\{\varphi_n\}_{n\ge 0}$ almost surely stabilize to a \emph{limit lamp configuration} $\varphi_{\infty}$ (see Definition \ref{def: stabilization of lamps}). If the random walk is non-degenerate, this implies the non-triviality of the associated Poisson boundary \cite[Propositions 6.1 and 6.4]{KaimanovcihVershik1983}. In contrast, if the induced random walk to $\Z^d$ is recurrent, the Poisson boundary is trivial \cite[Proposition 6.3]{KaimanovcihVershik1983}. Kaimanovich and Vershik asked whether, in the transient case and for finitely supported probability measures, the convergence to the limit lamp configuration completely describes the Poisson boundary. It is natural to extend the formulation of this question to arbitrary probability measure for which lamp configurations stabilize almost surely along sample paths of the $\mu$-random walk. 

Our first main result provides a description of the Poisson boundary of $A\wr B$, under general conditions on the measure $\mu$. 
\begin{thm} \label{thm: lamplighter Poisson boundary}
	Consider non-trivial countable groups $A$ and $B$. Let $\mu$ be a probability measure on $A\wr B$ with finite entropy and such that lamp configurations stabilize almost surely. Denote by $(\partial B,\nu_B)$ the Poisson boundary of the induced random walk on $B$. Then the Poisson boundary of $(A\wr B,\mu)$ is completely described by the space $A^B \times \partial B$, endowed with the corresponding hitting measure.
\end{thm}
If, in addition, the induced random walk on $B$ has a trivial Poisson boundary (in particular whenever $B=\Z^d$), Theorem \ref{thm: lamplighter Poisson boundary} implies the following.

\begin{cor}\label{cor: main cor of main theorem}
	Consider non-trivial countable groups $A$ and $B$. Let $\mu$ be a probability measure on $A\wr B$ with finite entropy and such that lamp configurations stabilize almost surely. Suppose that the Poisson boundary of the induced random walk on $B$ is trivial. Then the space of infinite lamp configurations $A^B$ endowed with the hitting measure is the Poisson boundary of $(A\wr B, \mu)$.
\end{cor}

The Poisson boundary of wreath products has been previously described under various conditions on the groups $A$ and $B$, as well as on the measure $\mu$ \cite{JamesPeres1996,Kaimanovich2001,KarlssonWoess2007,Sava2010,Erschler2011,LyonsPeres2021}; we explain these results in Subsection \ref{subsection: background}. These descriptions of the Poisson boundary use the space of limit lamp configurations, and the conditions on $\mu$ ensure that the stabilization phenomenon does occur. In particular, a sufficient condition for the stabilization of lamp configurations is that the measure $\mu$ has a finite first moment and that the induced random walk on $B$ is transient \cite[Lemma 1.1]{Erschler2011} (see also \cite[Theorem 3.3]{Kaimanovich1991} for $B=\Z^d$, and \cite[Theorem 2.9]{KarlssonWoess2007} for $B$ a free group). For the original groups studied by Kaimanovich and Vershik, Erschler proved that the Poisson boundary of a non-degenerate random walk on $A\wr \Z^d$, with $d\ge 5$, is equal to the space of limit lamp configurations, with the hypothesis that $A$ is finitely generated and that $\mu$ has a finite third moment \cite[Theorem 1]{Erschler2011}. Later, Lyons and Peres proved that this also holds for $d=3,4$, improving the assumption on $\mu$ by requiring only a finite second moment \cite[Theorem 5.1]{LyonsPeres2021}.  

Theorem \ref{thm: lamplighter Poisson boundary} generalizes all previous descriptions of the Poisson boundary of wreath products to a setting where there is no control over the tail decay of the measure $\mu$ and the previous techniques do not work. For our result to hold, we must add the hypothesis of stabilization of lamp configurations along sample paths. This is not always the case, and there are examples where this does not happen for measures $\mu$ with a finite $(1-\varepsilon)$-moment, for any $\varepsilon>0$ \cite[Proposition 1.1]{Kaimanovich1983} (see also \cite[Section 6]{Erschler2011} and the last paragraph of Section 5 in \cite{LyonsPeres2021}). Nonetheless, it is true that any non-degenerate random walk on $A\wr B$ with finite entropy and a transient projection to $B$ has a non-trivial Poisson boundary \cite[Theorem 3.1]{Erschler2004Liouville}. The description of the Poisson boundary seems more challenging in the case where there is no stabilization of the lamp configurations, as there are not even any known candidates for the Poisson boundary.

The following is a consequence of Theorem \ref{thm: lamplighter Poisson boundary}.

\begin{cor}\label{cor: finite first moment Zd poisson boundary} Let $A$ be a non-trivial finitely generated group, and let $d\ge 3$. Let $\mu$ be a non-degenerate probability measure on $A\wr \Z^d$ with a finite first moment. Then the Poisson boundary of $(A\wr \Z^d,\mu)$ is equal to the space of infinite lamp configurations endowed with the corresponding hitting measure.
\end{cor}
This was an open question, asked by Kaimanovich \cite[Example 3.6.7]{Kaimanovich2001} and by Lyons-Peres at the end of Section 5 in \cite{LyonsPeres2021}. We remark that \cite[Theorem 3.6.6]{Kaimanovich2001} partially answers this question for probability measures where the projection to $\Z^d$ has non-zero mean. This covers in particular all probability measures on $A\wr \Z$ with a finite first moment and a transient projection to $\Z$, and all probability measures on $A\wr \Z^2$ with a finite second moment and a transient projection to $\Z^2$. Indeed, the transience of the induced random walk on $\Z$ (resp.\ $\Z^2$) with a finite first moment (resp.\ finite second moment) hypothesis implies that the induced random walk on $\Z$ (resp.\ $\Z^d$) has non-zero mean \cite[Theorem 8.1]{Spitzer1976}.

Our second main theorem is an extension of Corollary \ref{cor: finite first moment Zd poisson boundary} to a more general setting, where the induced random walk on the base group $B$ does not necessarily have the Liouville property. Below, we denote by $\langle \supp{\mu}\rangle_{+}$ the semigroup generated by the support of $\mu$.
\begin{thm}\label{thm: Poisson for finite first moment}
	Let $A,B$ be finitely generated groups and let $\mu$ be a probability measure on $A\wr B$ such that $\langle \supp{\mu}\rangle_{+}$ contains two distinct elements with the same projection to $B$. Assume that $\mu$ has a finite first moment and that it induces a transient random walk on $B$. Then the space of infinite lamp configurations $A^B$ endowed with the corresponding hitting measure is equal to the Poisson boundary of $(A\wr B,\mu)$.
\end{thm}
The proofs of Theorems \ref{thm: lamplighter Poisson boundary} and \ref{thm: Poisson for finite first moment} use Kaimanovich's conditional entropy criterion \cite[Theorem 2]{Kaimanovich1985conditional}, \cite[Theorem 4.6]{Kaimanovich2000}. We refer to Subsection \ref{subsection: sketch of the proof} for a sketch of our argument. 

The key idea behind the proof of Theorem \ref{thm: Poisson for finite first moment} is that the extra hypothesis of finite first moment allows us to recover the Poisson boundary of the projected random walk on $B$ via the limit lamp configuration (Proposition \ref{prop: identification of B with finite first moment}). Note that the assumption that $\langle \supp{\mu}\rangle_{+}$ contains two distinct elements with the same projection to $B$ is necessary in order to do this, since otherwise the measure $\mu$ could be supported in a conjugate of the base group $B$ and then the limit lamp configuration does not encode information about the trajectory of the random walk in the base group (see Example \ref{example: conjugate base group}). 

It remains an open question whether one can recover the information about the Poisson boundary in the base group from the limit lamp configuration without the assumption of a finite first moment, which would improve Theorem \ref{thm: lamplighter Poisson boundary}.

\begin{question}
	Consider the hypotheses of Theorem \ref{thm: lamplighter Poisson boundary}, and suppose that the projected random walk to $B$ has a non-trivial Poisson boundary. Is the Poisson boundary of $(A\wr B,\mu)$ the space of limit lamp configurations $A^B$ endowed with the corresponding hitting measure?
\end{question}

We finish this subsection with an application of our main theorems to groups other than wreath products. Indeed, Theorem \ref{thm: lamplighter Poisson boundary} does not have a non-degeneracy assumption, and Theorem \ref{thm: Poisson for finite first moment} holds for a large class of possibly degenerate probability measures with a finite first moment, so that both of these results can be applied for random walks supported on proper subgroups of wreath products. In particular, this allows us to identify the Poisson boundary of random walks on free solvable groups, with adequate assumptions on the step distribution of the random walk.

Let $F_d$ be a free group of rank $d\ge 1$, consider $N\lhd F_d$ a normal subgroup and let us denote by $[N,N]$ the subgroup of $F_d$ generated by all commutators $[n_1,n_2]\coloneqq n_1n_2n_1^{-1}n_2^{-1}$ of arbitrary elements $n_1,n_2\in N$. In what follows we consider groups of the form $F_d/[N,N]$. Let us introduce the notation $F_d^{(0)}\coloneqq F_d$ and $F_d^{(k)}\coloneqq\left[F^{(k-1)}_d,F^{(k-1)}_d\right]$ for $k\ge 1$. Then, choosing $N=F_d^{(k)}$ in the above construction gives rise to the group $S_{d,k}\coloneqq F_d/F_d^{(k)}$, which is called the \emph{free solvable group of rank $d$ and class $k$}. Note that every solvable group of class $k$ generated by $d$ elements is a quotient of the group $S_{d,k}$. For $k=1$ one obtains $S_{d,1}=\Z^d$ the free abelian group of rank $d$, and in the case $k=2$ the group $S_{d,2}$ is commonly referred to as the \emph{free metabelian group of rank $d$}.

 The well-known \textit{Magnus embedding} expresses the group $F_d/[N,N]$ as a subgroup of the wreath product $\Z^d\wr F_d/N$. This result was proved by Magnus \cite{Magnus1939}, who described an explicit embedding of the group $F_d/[N,N]$ into a matrix group with coefficients in a module over the group ring $\Z(F_d/N)$, which is isomorphic to $\Z^d\wr F_d/N$. This embedding together with its description via Fox's free differential calculus \cite{FoxI,FoxII,FoxIII,FoxIV,FoxV} have been extensively used in the past decades to study algorithmic and geometric properties of free solvable groups; see e.g.\ \cite{RemeslennikovSokolov1970,DromsLewinServatius1993,Vershik2000freesolvable,MyasnikovRomankovUshakovVershik2010,Vassileva2011,Vassileva2012,Ushakov2014,Sale2015}. 
 
 Let $X\subseteq F_d/N$ be the image of a free generating set of $F_d$ via the canonical quotient epimorphism and let us consider the labeled, directed, Cayley graph $\mathrm{Cay}(F_d/N)$ of the group $F_d/N$ with respect to $X$. Denote by $\mathrm{Edges}(F_d/N)$ the set of labeled and directed edges of $\mathrm{Cay}(F_d/N)$. As explained in \cite{DromsLewinServatius1993,Vershik2000freesolvable,MyasnikovRomankovUshakovVershik2010,SaloffCosteZheng2015}, the Magnus embedding together with a geometric interpretation of Fox derivatives gives a way of associating with every group element $g\in F_d/[N,N]$ a unique finitely supported \emph{flow} $\mathfrak{f}_g:\mathrm{Edges}(F_d/N)\to \Z$. We recall this description in more detail in Subsections \ref{subsection: magnus} and \ref{subsection: flows}. In the context of random walks on groups, this interpretation has been used for the description of the Poisson boundary in the case of free metabelian groups \cite{Erschler2011, LyonsPeres2021}, and for the computation of the asymptotics of the return probability of symmetric finitely supported probability measures on free solvable groups in \cite{SaloffCosteZheng2015}. In particular, it is proved in \cite[Lemma 1.2]{Erschler2011} that if $\mu$ is a probability measure on $F_d/[N,N]$ with a finite first moment and that induces a transient random walk on $F_d/N$, then for almost every sample path $\{w_n\}_{n\ge 0}$ of the $\mu$-random walk on $F_d/[N,N]$, the associated flows $\{\mathfrak{f}_{w_n}\}_{n\ge 0}$ stabilize to a limit function $\mathfrak{f}_{\infty}:\mathrm{Edges}(F_d/N)\to \Z$. Hence, under the previous hypotheses on $\mu$, the space $\Z^{\mathrm{Edges}(F_d/N)}$ endowed with the corresponding hitting measure is a $\mu$-boundary for $F_d/[N,N]$. The space $\Z^{\mathrm{Edges}(\Z^d)}$ is known to be the Poisson boundary of the free metabelian group $S_{d,2}$ of rank $d\ge 5$ for adapted probability measures $\mu$ with a finite third moment \cite[Theorem 2]{Erschler2011}, and for $d\ge 3$ and $\mu$ with a finite second moment \cite[Section 6]{LyonsPeres2021}. 
  The following is the principal result that we obtain in this context.
 
 \begin{cor}\label{cor: general description for general Magnus embedding}
 	Let $N$ be a normal subgroup of a free group $F_d$, and consider a probability measure $\mu$ on $F_d/[N,N]$. 
 	\begin{enumerate}
 		\item Suppose that $H(\mu)<\infty$ and that the flows associated with the $\mu$-random walk on $F_d/[N,N]$ stabilize almost surely. Denote by $(\partial (F_d/N),\nu_{F_d/N})$ the Poisson boundary of the random walk induced by $\mu$ on $F_d/N$. Then the Poisson boundary of $(F_d/[N,N],\mu)$ is the space $\Z^{\mathrm{Edges}(F_d/N)}\times \partial (F_d/N)$, endowed with the corresponding hitting measure. 
 		
 		\item Suppose that $\mu$ has a finite first moment. Assume furthermore that $\langle \supp{\mu}\rangle_{+}$ contains two distinct elements with the same projection to $F_d/N$, and that $\mu$ induces a transient random walk on $F_d/N$. Then the Poisson boundary of $(F_d/[N,N],\mu)$ is the space $\Z^{\mathrm{Edges}(F_d/N)}$, endowed with the corresponding hitting measure.
 	\end{enumerate}

 \end{cor}
When $N=F^{(k)}_d$ is the $k$-th derived subgroup of $F_d$, Theorem \ref{cor: general description for general Magnus embedding} states that for any probability measure $\mu$ on the free solvable group $S_{d,k}$ such that $H(\mu)<\infty$ and for which flows stabilize almost surely along sample paths of the $\mu$-random walk, the Poisson boundary of $(S_{d,k},\mu)$ is the space $\Z^{\mathrm{Edges}(S_{d,k-1})}\times \partial(S_{d,k-1})$ endowed with the hitting measure. In particular, if $\mu$ has a finite first moment then $H(\mu)<\infty$, and thanks to \cite[Lemma 1.2]{Erschler2011} the flows associated with the $\mu$-random walk stabilize almost surely, so that Theorem \ref{cor: general description for general Magnus embedding} applies. With this, we obtain the description of the Poisson boundary for all random walks on any free solvable group with a step distribution with a finite first moment.
 \begin{cor}\label{cor: description for free solvable group}
 	Let $S_{d,k}$ be the free solvable group of rank $d\ge 2$ and derived length $k\ge 2$, and let $\mu$ be an adapted probability measure on $S_{d,k}$ with a finite first moment.
 	\begin{enumerate}
 		
 		\item  If $d=k=2$, then the Poisson boundary of $(S_{2,2},\mu)$ is non-trivial if and only if the projection of $\mu$ to $\Z^2$ induces a transient random walk In this case, the Poisson boundary is the space $\Z^{\mathrm{Edges}(\Z^2)}$, endowed with the corresponding hitting measure.
 		\item  If $\max\{d,k\}\ge 3$, then the Poisson boundary of $(S_{d,k},\mu)$ is non-trivial and it is the space $\Z^{\mathrm{Edges}(S_{d,k-1})}$, endowed with the corresponding hitting measure.
 	\end{enumerate}
 \end{cor}
This corollary generalizes the results of \cite[Theorem 2]{Erschler2011} for free metabelian groups $S_{d,2}$ with $d\ge 5$ and $\mu$ with a finite third moment, and of \cite[Section 6]{LyonsPeres2021} for $S_{d,2}$, with $d\ge 3$ and $\mu$ with a finite second moment. Additionally, \cite[Section 6]{LyonsPeres2021} also describes the Poisson boundary for finitely supported and adapted probability measures $\mu$ on groups of the form $F_d/[N,N]$, in the case where $F_d/N$ has at least cubic growth.
 
\subsection{Background}\label{subsection: background}
The interest in considering in Theorem \ref{thm: lamplighter Poisson boundary} a class of infinitely supported probability measures with no extra moment assumptions, in particular with an infinite first moment, comes from the fact that such measures appear naturally in the study of Poisson boundaries. One such occurrence is in the study of groups of intermediate growth. Random walks on groups of intermediate growth with a non-trivial Poisson boundary and a known control of the tail decay of the step distribution have been constructed in \cite[Theorem 2]{Erschler2004}, and more recently in \cite[Theorems A, B and C]{ErschlerZheng2020} where in particular near-optimal lower bounds for the volume growth of Grigorchuk's group are obtained. Such measures must necessarily have an infinite first moment, since the entropy criterion shows that any probability measure with a finite first moment in a group of subexponential growth has a trivial Poisson boundary.  We refer to \cite[Section 2.1]{ErschlerZheng2020} and \cite[Section 2]{Zheng2022} for a more detailed explanation of the relation between random walks with a non-trivial Poisson boundary and estimates for the growth of a group. Additionally, we mention that groups of exponential growth for which every finitely supported measure has a trivial Poisson boundary are constructed in \cite[Theorem A]{BartholdiErschler2017}. It is currently an open problem whether there exist groups of exponential growth for which every probability measure with a finite first moment has a trivial Poisson boundary \cite[Question 1.1]{BartholdiErschler2017}.

Another setting where infinitely supported measures appear naturally is in the study of amenable groups. On the one hand, the amenability of a group is entirely determined by the existence of random walks with non-trivial boundaries. More precisely, every non-degenerate random walk in a non-amenable group has a non-trivial Poisson boundary \cite[Proposition II.1]{Azencott1970} (see also the last two paragraphs of Section 9 in \cite{Furstenberg1973}), and every amenable group $G$ admits a non-degenerate symmetric measure $\mu$ with $\supp{\mu}=G$ and with a trivial Poisson boundary \cite[Theorem 1.10]{Rosenblatt1981}, \cite[Theorem 4.4]{KaimanovcihVershik1983}. On the other hand, the class of amenable groups that admit non-degenerate measures with a non-trivial Poisson boundary coincides with the class of amenable groups that are not hyper-FC-central \cite[Theorem 1]{FrischHartmanTamuzVahidi2019} (the fact that any random walk on a hyper-FC central countable group has a trivial Poisson boundary had previously been proved in \cite{LinZaidenberg1998,Jaworski2004}). We recall that a group is hyper-FC-central if and only if none of its quotients has the infinite conjugacy class (ICC) property. The result of \cite{FrischHartmanTamuzVahidi2019} is further developed in \cite[Theorem A]{ErschlerKaimanovich2023}, where it is proved that any countable ICC group $G$ admits symmetric non-degenerate measures, such that the associated non-trivial Poisson boundary can be described in terms of the convergence to the boundary of a locally finite forest, whose vertex set is $G$. All of the correspondences mentioned above fail if we restrict ourselves to probability measures with a finite first moment. Indeed, the group $\Z/2\Z\wr \Z^3$ is amenable, but every non-degenerate measure with a finite first moment (and more generally, with finite entropy) has a non-trivial Poisson boundary. Additionally, finitely generated groups of intermediate growth provide examples of ICC groups in which every probability measure with a finite first moment has a trivial Poisson boundary.

The known descriptions of the Poisson boundary for non-amenable groups often rely on the group possessing some hyperbolic-like nature. This is the case for free groups \cite{DynkinMaljutov1961,Derrienic1975}, hyperbolic groups \cite{Ancona1987}, \cite[Theorems 7.4 and 7.7]{Kaimanovich2000}, \cite[Theorem 1.1]{ChawlaForghaniFrischTiozzo2022}, and more generally acylindrically hyperbolic groups \cite[Theorem 1.5]{MaherTiozzo2018}, \cite[Theorem 1.2]{ChawlaForghaniFrischTiozzo2022}. These papers cover previously known descriptions for some specific families of groups; see \cite{Woess1989,KaimanovichMasur1996,FarbMasur1998,GauteroMatheus2012,Horbez2016}. Another family of groups in this category are discrete subgroups of semi-simple Lie groups \cite{Furstenberg1971,Ledrappier1985}, \cite[Theorems 10.3 and 10.7]{Kaimanovich2000}.

The most relevant descriptions of (non-trivial) Poisson boundaries for our results are those of wreath products, which provided the first examples of amenable groups that admit symmetric random walks with a non-trivial Poisson boundary \cite[Proposition 6.1]{KaimanovcihVershik1983}. We now make a summary of the known complete descriptions of Poisson boundaries of wreath products.
\begin{itemize}[--]
	\item The first complete description of a non-trivial Poisson boundary of a random walk on a wreath product was by James-Peres \cite[Corollary 1.1]{JamesPeres1996}, who considered some particular degenerate measures on $\Z\wr\Z^d$, $d\ge 1$. In general, one can associate with every random walk on a finitely generated group $B$ a degenerate random walk on $\Z\wr B$ that uses the lamp group $\Z$ to keep track of the number of visits to each point in the base group $B$. Then, the Poisson boundary of this random walk on $\Z\wr B$ is equal to the \emph{exchangeability boundary} of the induced random walk on $B$. James and Peres proved that the Poisson boundary of these specially chosen measures on $\Z\wr \Z^d$ is completely described by the number of visits to points in the base group. We refer to \cite{JamesPeres1996}, \cite[Section 6]{Kaimanovich1991} and \cite[Subsection 4.1]{Erschler2011} for the definition of the exchangeability boundary and more details on this correspondence. James and Peres' description of the Poisson boundary of $\Z\wr\Z^d$ for such degenerate random walks is based on the fact that transient random walks on $\Z^d$ with a finitely supported step distribution have infinitely many cutpoints with probability $1$ \cite[Theorem 1.2]{JamesPeres1996}.
	\item In the case of non-degenerate measures, Kaimanovich proved the following: for any probability measure on $A\wr \Z^d$, $d\ge 1$, with a finite first moment and such that the induced random walk on $\Z^d$ has non-zero drift, the Poisson boundary is equal to the space of limit lamp configurations, endowed with the corresponding hitting measure \cite[Theorem 3.6.6]{Kaimanovich2001}. This result is a consequence of the strip criterion, whose hypotheses are verified by using the fact that the induced random walk on the base group $\Z^d$ has a linear rate of escape. This technique does not work if the induced random walk on $\Z^d$ is centered (which can only occur for $d\ge 3$ in the case of non-trivial Poisson boundaries).
	\item There are other base groups, distinct from $\Z^d$, for which the strip criterion was applied to provide a description of the Poisson boundary for non-degenerate measures with a finite first moment in the corresponding wreath product: this has been done for finite lamp groups $A$  and for $B$ a (non-abelian) free group \cite[Theorem 3.2]{KarlssonWoess2007}, and for $B$ non-elementary hyperbolic or with infinitely many ends \cite[Theorems 4.2 and 4.3]{Sava2010}.
	\item Returning to the case where the base group is $\Z^d$, Erschler proved that for $A$ finitely generated, $d\ge 5$, and for a probability measure $\mu$ on $A\wr \Z^d$ with a finite third moment and that induces a transient random walk on $\Z^d$, the associated Poisson boundary is equal to the space of limit lamp configurations, endowed with the hitting measure \cite[Theorem 1]{Erschler2011}. Her proof consists of using the support of the limit lamp configuration together with cut-balls of the trajectory in the base group to guess which points of it are visited at certain time instants. With this, the identification of the Poisson boundary is obtained by using a version of Kaimanovich's ray criterion.
	\item The latter result was extended by Lyons-Peres to $d\ge 3$, and for measures $\mu$ with a finite second moment and that induce a transient random walk on $\Z^d$ \cite[Theorem 5.1]{LyonsPeres2021}. In their proof, the second moment assumption controls how often lamps are modified very far away from the current position, and an argument with cut-spheres together with Kaimanovich's conditional entropy criterion concludes the proof. In addition, they show that the space of limit lamp configurations describes the Poisson boundary of random walks on groups $A\wr B$ with $A$ finite and $B$ finitely generated, for step distributions with finite entropy, that induce a transient random walk on the base group $B$, and with bounded lamp range (i.e.\ at each step of the random walk, the lamp configuration $\varphi_n$ is modified in a uniformly bounded neighborhood of the current base position $X_n$) \cite[Theorem 1.1]{LyonsPeres2021}. Lyons-Peres also prove an enhanced version of Kaimanovich's conditional entropy criterion \cite[Corollary 2.3]{LyonsPeres2021}, and use it to give a short proof of the description of the Poisson boundary for simple random walks on $A\wr \Z^d$, $d\ge 3$, by using cut-spheres to estimate the values of the lamp configuration at a given instant \cite[Theorem 3.1]{LyonsPeres2021}. In particular, the results of Lyons-Peres cover the description of the Poisson boundary for simple random walks on groups $A\wr B$ where $A$ is finite and $B$ is finitely generated, and recover the results that had been obtained previously in the literature for finitely supported random walks.
\end{itemize}

Our method for proving Theorem \ref{thm: lamplighter Poisson boundary} is different from the results cited previously. Since there are no moment assumptions in our hypotheses, the estimates we obtain for the lamp configuration at some given instant rely uniquely on the fact that stabilization does occur. Indeed, for the (possibly heavy-tailed) measures we consider, there may be frequent modifications to the lamp configuration at distances arbitrarily far away from the base position. For these measures, the method of cut-balls used by Erschler for measures with a finite third moment does not apply, and neither does the technique of cut-spheres used by Lyons-Peres for measures with a finite second moment. We also mention that our proof uses the original version of the conditional entropy criterion of Kaimanovich, and does not require the enhanced version proved by Lyons-Peres. The technique we use to prove our theorems has some commonalities with the \emph{pin-down approximation} used by \cite{ChawlaForghaniFrischTiozzo2022} for (acylindrically) hyperbolic groups. In both arguments, the entropy of the random walk position at some time instant is estimated by conditioning on partitions of the space of trajectories that add a small amount of information.  The core part of the proof of \cite[Theorems 1.1 and 1.2]{ChawlaForghaniFrischTiozzo2022} is estimating the word length of the random walk element at a given instant, and it relies on the fact that a sample path of the random walk on a hyperbolic group can be divided into consecutive segments, each separated from the next one by a \emph{pivot} of the trajectory. The position of the random walk is then deduced from the additional partitions associated with the time instants in which pivots occur, together with the distance between consecutive pivots and the increments from long time intervals in which there were no pivots. In our case, the argument is different. The most important part of our proof corresponds to Steps \prange{step: step 7}{step: step 10} in the sketch below, where the stabilization of lamp configurations allows us to reveal the lamp increments that were done in group elements whose associated lamp configuration has not yet stabilized to its limit value, at a given time instant.

We now sketch the proof of our main theorems.

\subsection{Sketch of the proof of Theorems \ref{thm: lamplighter Poisson boundary} and \ref{thm: Poisson for finite first moment}}\label{subsection: sketch of the proof}
Thanks to Kaimanovich's conditional entropy criterion (Theorem  \ref{thm: conditional entropy alternative formulation}), in order to prove Theorem \ref{thm: lamplighter Poisson boundary} (resp.\ Theorem \ref{thm: Poisson for finite first moment}) it suffices to show that for every $\varepsilon>0$, the mean conditional entropy of the random walk conditioned on hitting boundary points $(\varphi_{\infty},\xi)\in A^B\times \partial B$ (resp. $\varphi_{\infty}\in A^B$) is at most $\varepsilon n+K$, where $K$ is a constant independent of $n$. The idea is that, conditioned on the boundary point $(\varphi_{\infty},\xi)$ (resp.\ $\varphi_{\infty}$), we can recover the position of the random walk at time $n$ by adding information in the form of countable partitions of the space of trajectories, that in average add a small amount of entropy to our process.

For every $\varepsilon>0$, we follow the next steps.
\begin{enumerate}
	\item\label{step: step 1} The first step of the proof is to observe that the boundary point allows us to determine the position in $B$ of the $\mu$-random walk every $t_0$ steps, for $t_0$ large enough, at the cost of adding a small amount of entropy. We save this information in the partition $\mathcal{P}_n^{t_0}$ and call it the \emph{$t_0$-coarse trajectory} (Definition \ref{def: coarse trajectory in base group}). In the case of Theorem \ref{thm: lamplighter Poisson boundary}, the above is done by conditioning on the boundary point $\xi \in \partial B$ (Lemma \ref{lem: coarse traj has small entropy}), whereas for Theorem \ref{thm: Poisson for finite first moment} we do it by conditioning on the limit lamp configuration $\varphi_{\infty}$, and use the additional hypothesis of a finite first moment (Proposition \ref{prop: identification of B with finite first moment}). This is the only place where there is a difference between the proofs of Theorems \ref{thm: lamplighter Poisson boundary} and \ref{thm: Poisson for finite first moment}, and the rest of the sketch below is common to both of them.
	
	\item \label{step: step 2} The partition $\mathcal{P}_n^{t_0}$ contains the position in $B$ in the last instant that is a multiple of $t_0$. Then, getting to the time instant $n$ adds only a constant amount of entropy. Because of this, throughout the rest of the sketch we will assume that $n$ is a multiple of $t_0$, so that we do not have to mention this argument multiple times. 
	
	In what follows, we explain our approach to estimating the entropy of the lamp configuration at time $n$. 
	
	\item \label{step: step 3} We introduce the notion of \emph{$(R,L)$-good elements}, which are the group elements that only do movements in the base group or lamp modifications inside some finite neighborhood $R\subseteq B$ of the current position in $B$, and which modify lamp configurations by a state inside some finite set $L\subseteq A$ (Definition \ref{def: bad elements and bad intervals}). An element is said to be \emph{$(R,L)$-bad} if it is not $(R,L)$-good. Note that if $A$ and $B$ are finitely generated, then so is $A\wr B$, and our definition of an element being $(R,L)$-good is equivalent to having a uniform upper bound (depending on $R$ and $L$) for its word length. 
	
	\item \label{step: step 4} We divide the time interval $[0,n]$ into subintervals $I_j$, $j=1,\ldots, n/t_0$, of length $t_0$ (Definition \ref{def: time intervals}). We say that an interval $I_j$ is \emph{$(R,L)$-bad} if some increment $g_i$, $i\in I_j$, is $(R,L)$-bad (Definition \ref{def: bad elements and bad intervals}). We save the information of the increments done during an $(R,L)$-bad interval in the partition $\beta_n(t_0,R,L)$, and call it the collection of \emph{$(t_0,R,L)$-bad increments} (Definition \ref{defn: define the partition of bad and final intervals.}).
	
	\item \label{step: step 5} We then prove that when $R$ and $L$ are large enough, the partition $\beta_n(t_0,R,L)$ has small entropy (Lemma \ref{lem: bad increments have small entropy}). From this, we can estimate the values of the lamp configuration at instant $n$ for the elements of $B$ as follows.
	
	\item \label{step: step 6} Define the \emph{coarse neighborhood} $\mathcal{N}_n(t_0,R)$ of the trajectory as a neighborhood of $\mathcal{P}_n^{t_0}$ (in the base group $B$) such that any modification to the lamp configuration outside of $\mathcal{N}_n(t_0,R)$ must have occurred during an $(R,L)$-bad interval (Definition \ref{def:neighborhood of coarse traj}). Then the values of the lamp configuration at elements outside $\mathcal{N}_n(t_0,R)$ can be deduced from the $t_0$-coarse trajectory $\mathcal{P}_n^{t_0}$ together with the information of the $(t_0,R,L)$-bad increments $\beta_n(t_0,R,L)$ (Lemma \ref{lem: lamp config outside coarse neighborhood has small entropy}).
	
	\item \label{step: step 7} The most important part of the proof is to prove that the entropy of the lamp configuration inside the coarse neighborhood  $\mathcal{N}_n(t_0,R)$ is low when conditioned on the limit lamp configuration $\varphi_{\infty}$, the $t_0$-coarse trajectory $\mathcal{P}_n^{t_0}$ and the $(t_0,R,L)$-bad increments $\beta_n(t_0,R,L)$ (Proposition \ref{prop: the lamp config INSIDE the neighborhood has small entropy}). Here is where we use the hypothesis that the lamp configuration stabilizes almost surely. This is the objective of Section \ref{section: entropy inside the coarse neighborhood}. In Subsection \ref{subsection: proof of lamps inside for FINITE lamp group} we present a proof of Proposition \ref{prop: the lamp config INSIDE the neighborhood has small entropy} in the case where the lamp group $A$ is finite. The case of infinite lamp groups has some additional difficulties, and the idea of our proof is the following.
	
	\item \label{step: step 8} In expectation, the lamp configuration of most of the elements inside the $t_0$-coarse neighborhood $\mathcal{P}_n^{t_0}$ will be stabilized to the value assigned by the limit lamp configuration $\varphi_{\infty}$ (Lemma \ref{lem: expectation of unstable points is small}). Hence, revealing the ones that are not yet stabilized has low entropy (Lemma \ref{lem: unstable points have small entropy}). Furthermore, if $t_0$ is large enough, then revealing the instants where these increments were applied also has low entropy (Lemma \ref{lem: visit times to unstable points have small entropy}).
	
	\item \label{step: step 9} If these increments were done during an $(R,L)$-bad interval, we already know their value since this information is contained in the partition $\beta_n(t_0,R,L)$. On the other hand, if they were done during an $(R,L)$-good interval, we can bound the possible number of values they can take in terms of the size of $R$ and $L$. With this, we see that the $(R,L)$-good lamp increments done at every position whose lamp configuration has not yet stabilized contain a small amount of entropy (Lemma \ref{lem: increments at unstable points have small entropy}).
	
	\item \label{step: step 10} The lamp configuration inside the coarse neighborhood $\mathcal{N}_n(t_0,R)$ is completely determined by the above increments together with the information of the $(t_0,R,L)$-bad intervals $\beta_n(t_0,R,L)$ and the $t_0$-coarse trajectory $\mathcal{P}_n^{t_0}$. Together with Step (\labelcref{step: step 6}), we conclude that the mean conditional entropy with respect to the limit lamp configuration $\varphi_{\infty}$, conditioned on the $t_0$-coarse trajectory $\mathcal{P}_n^{t_0}$ and the $(t_0,R,L)$-bad increments $\beta_n(t_0,R,L)$ has low entropy (Proposition \ref{prop: partitions pin down the n-th instant}).
\end{enumerate}

\subsection{Organization} In Section \ref{section: preliminaries} we recall some generalities on Poisson boundaries and entropy. Then in Section \ref{section: the coarse trajectory} we give the setup for Theorem \ref{thm: lamplighter Poisson boundary} and prove the entropy estimates of Steps \prange{step: step 1}{step: step 6} for partitions of the space of trajectories up to instant $n$ mentioned in the sketch of the proof above. The most important part of our argument is the entropy estimate described in Steps \prange{step: step 7}{step: step 10} above, and it is presented in Section \ref{section: entropy inside the coarse neighborhood}. The key result of this section is Proposition \ref{prop: the lamp config INSIDE the neighborhood has small entropy}. To prove it, we introduce additional partitions related to the group elements for which the lamp configuration has not stabilized to its limit value by time $n$, and which do not play a role outside this section. Next, in Section \ref{section: proofs of the main theorems} we present the proofs of Theorems \ref{thm: lamplighter Poisson boundary} and \ref{thm: Poisson for finite first moment}, as well as Corollary \ref{cor: finite first moment Zd poisson boundary}. Finally, in Section \ref{section: free solvable groups} we prove Corollaries \ref{cor: general description for general Magnus embedding} and \ref{cor: description for free solvable group} for groups of the form $F_d/[N,N]$ and free solvable groups.

\subsection*{Acknowledgments} The first named author was supported by NSF Grant DMS-2102838. The second named author has received funding from the European Union’s Horizon 2020 research and innovation programme under the Marie Sk\l{}odowska-Curie grant agreement N\textsuperscript{\underline{o}} 945322, and from the European Research Council (ERC) under the European Union’s Horizon 2020 research and innovation program (grant agreement N\textsuperscript{\underline{o}} 725773). We would like to thank Anna Erschler for bringing this problem into our attention, for many helpful discussions and for her comments and suggestions. We thank Russel Lyons for his careful reading and helpful comments on the first versions of this paper. We would also like to thank Vadim Kaimanovich for helpful discussions around the conditional entropy criterion. We are grateful to the anonymous referee for their useful comments which improved the exposition of the paper.
\section{Preliminaries}\label{section: preliminaries}
\subsection{Wreath products and lamplighter groups}
For $A,B$ groups, we define their \textit{wreath product} $A\wr B$ as the semidirect product $\bigoplus_{B} A \rtimes B$, where $\bigoplus_{B} A $ is the group of finitely supported functions $f:B\to A$ endowed  with the operation $\oplus$ of componentwise multiplication. We denote by $\supp{f}$ the finite subset of $B$ to which $f$ assigns non-trivial values. Here, the group $B$ acts on the direct sum $\bigoplus_{B} A $ from the left by translations. That is, for $f:B\to A$, and any $b\in B$ we have $(b\cdot f)(x)=f(b^{-1}x), \ x\in B.$ Elements of $A\wr B$ can be expressed as a tuple $(f,b)$, where $f\in \bigoplus_{B} A$ and $b\in B$. The product between two elements elements $(f,b)$, $(f^{\prime},b^{\prime})\in A\wr B$ is given by $(f,b)\cdot(f^{\prime},b^{\prime})= (f\oplus (b\cdot f^{\prime}),bb^{\prime}).$ There is a natural embedding of $B$ into $A\wr B$ via $B\to A\wr B,\ b\mapsto (\mathds{1},b)$, where $\mathds{1}(x)=e_A$ for any $x\in B$. Similarly, we can embed $A$ into $A\wr B$ via $A\to A\wr B, a\mapsto (\delta^{a}_{e_B},e_B)$, where $\delta^{a}_{e_B}(e_B)=a$ and $\delta^{a}_{e_B}(x)=e_A$ for any $x\neq e_B$. If $A$ and $B$ are generated by finite sets $S_A$ and $S_B$, their copies inside $A\wr B$ through the above embeddings generate $A\wr B$.

Wreath products of the form $\Z/2\Z\wr B$ are also called ``lamplighter groups'', due to the following interpretation: one pictures the Cayley graph $\cay{B}{S_B}$ as a street with lamps at every element, each of which can be in two different states given by the elements of $\Z/2\Z$. An element $g=(f,b)\in \Z/2\Z\wr B$ consists of a \emph{lamp configuration} $f\in \bigoplus_{B} \Z/2\Z$, which encodes the state of the lamp at each element of $B$, together with a position $b\in B$, which corresponds to a person standing next to a particular lamp. Multiplying on the right by elements of $\Z/2\Z$ corresponds to switching the state of the lamp at position $b$, whereas multiplying on the right by an element of $B$ changes the position of the person without modifying the lamp configuration.

\subsection{Random walks and the Poisson boundary}\label{subsection: rw and poisson boundary}
Let $G$ be a countable group and let $\mu$ be a probability measure on $G$. The \emph{$\mu$-random walk} on $G$ is the Markov chain $\{w_n\}_{n\ge 0}$ defined by $w_0=e_G$, and for $n\ge 1$, $w_n=g_1g_2\cdots g_n,$ where $\{g_i\}_{i\ge 1}$ is a sequence of independent, identically distributed, random variables on $G$ with law $\mu$. The law $\P$ of the process $\{w_n\}_{n\ge 1}$ is defined as the push-forward of the Bernoulli measure $\mu^{\mathbb{N}}$ through the map
\begin{equation*}
	\begin{aligned}
		G^{\mathbb{N}}&\to G^{\mathbb{N}}\\
		(g_1,g_2,g_3,\ldots)&\mapsto (w_1,w_2,w_3,\ldots)\coloneqq (g_1,g_1g_2,g_1g_2g_3,\ldots).
	\end{aligned}
\end{equation*}
The space $(G^{\mathbb{N}},\P)$ is called the \emph{path space} or the \emph{space of trajectories} of the $\mu$-random walk. The Poisson boundary of $(G,\mu)$ is a probability measure space $(B,\nu)$ that is endowed with a measurable action of $G$ and a $G$-equivariant map $\mathbf{bnd}:(G^{\mathbb{N}},\P)\to (B,\nu)$ such that
\begin{enumerate}
	\item\label{cond2: Poisson bdry}  $\mathbf{bnd}\circ T=\mathbf{bnd}$, where $T:G^{\mathbb{N}}\to G^{\mathbb{N}}$ is the \emph{shift map on the space of trajectories} defined by $T(w_1,w_2,w_3,\ldots)\coloneqq (w_2,w_3,\ldots)$ for $(w_1,w_2,\ldots)\in G^{\mathbb{N}}$, and
	\item \label{cond3: Poisson bdry} $\nu=\mathbf{bnd}_{*}\P$.
\end{enumerate} These conditions imply that $\nu$ is a \emph{$\mu$-stationary} measure, meaning that it satisfies the equation $\nu=\mu*\nu\coloneqq\sum_{g\in G}\mu(g)g_{*}\nu$. In general, any probability $G$-space $(X,\lambda)$ endowed with a boundary map $(G^{\mathbb{N}},\P)\to (X,\lambda)$ that satisfies Conditions \eqref{cond2: Poisson bdry} and \eqref{cond3: Poisson bdry} is called a \emph{$\mu$-boundary} of $G$. The Poisson boundary is unique up to a $G$-equivariant isomorphism and it is maximal, in the sense that the boundary map $(G^{\mathbb{N}},\P)\to (X,\lambda)$ of any $\mu$-boundary $(X,\lambda)$ factors through $(B,\nu)$. 

In this paper we will work with Definition \ref{defn: Poisson boundary original def} for the Poisson boundary. Alternatively, it can also be defined as the space of ergodic components of the shift map $T:G^{\mathbb{N}}\to G^{\mathbb{N}}$ on the space of trajectories, where $T(w_1,w_2,w_3,\ldots)\coloneqq (w_2,w_3,\ldots)$ for $(w_1,w_2,\ldots)\in G^{\mathbb{N}}$. For further equivalent definitions the Poisson boundary, we refer to \cite{KaimanovcihVershik1983}, \cite[Section 1]{Kaimanovich2000} and the references therein. We also refer to the surveys \cite{Furman2002,Erschler2010} for an overview of the study of Poisson boundaries and to the survey \cite{Zheng2022} for more recent applications of random walks and Poisson boundaries to group theory.

\subsection{Entropy}\label{subsection: entropy}

The entropy of a probability measure $\mu$ on the group $G$ is defined as $H(\mu)\coloneqq -\sum_{g\in G}\mu(g)\log(\mu(g))$. More generally, let $\rho=\{\rho_k\}_{k\ge 1}$ be a countable partition of the space of sample paths $G^{\mathbb{N}}$. The entropy of $\rho$ with respect to the measure $\P$ is defined as $
	H(\rho)\coloneqq -\sum_{k\ge 1}\P(\rho_k)\log \P(\rho_k).$

The following sequence of partitions will appear in the formulation of Theorem \ref{thm: conditional entropy criterion - Kaimanovich}.

\begin{defn}\label{def: partition rw at time n}
	For every $n\ge 1$, define the partition $\alpha_n$ of the space of sample paths $G^{\mathbb{N}}$, where two trajectories $\mathbf{w},\mathbf{w^\prime}\in G^{\mathbb{N}}$ belong to the same element of $\alpha_n$ if and only if $w_n=w^\prime_n$.
\end{defn} 

The \emph{entropy criterion} of Derriennic \cite{Derrienic1980} and Kaimanovich-Vershik \cite{KaimanovcihVershik1983} states that if $\mu$ is a probability measure on $G$ with $H(\mu)<\infty$, then the Poisson boundary of $(G,\mu)$ is trivial if and only if $h_{\mu}=0.$ In what follows we recall the basic definitions of the conditional probability measures and conditional entropy associated with a $\mu$-boundary, and then state the \emph{conditional entropy criterion} of Kaimanovich below, that characterizes whether a $\mu$-boundary is the Poisson boundary.

\subsection{The conditional entropy criterion}

Let $\mathbf{X}=(X,\lambda)$ be a $\mu$-boundary of $G$. Using Rokhlin's theory of measurable partitions of Lebesgue spaces, the probability measure $\P$ can be disintegrated with respect to the map $\mathbf{bnd}:(G^{\mathbb{N}},\P)\to (X,\lambda)$. That is, for $\lambda$-a.e.\ $\xi \in X$,  there exists the conditional probability measure $\P^{\xi}$ on $G^{\mathbb{N}}$ that is supported on the fiber $\mathbf{bnd}^{-1}(\{\xi\})$, and it holds that 	$\P=\int_{X}\P^{\xi}\ d\lambda (\xi)$; see \cite[Section I.7]{Rohlin1967} and \cite[Section 3]{Kaimanovich2000}.

\begin{defn}\label{defn: average conditional entropy}
	Consider $\rho=\{\rho_k\}_{k\ge 1}$ a countable partition of the space of sample paths $G^{\mathbb{N}}$ and let $\mathbf{X}=(X,\lambda)$ be a $\mu$-boundary of $G$. For $\lambda$-a.e.\ $\xi\in X$, we define the \emph{conditional entropy of $\rho$ given $\xi$} as
	\begin{equation*}
		H_{\xi}(\rho)\coloneqq-\sum_{k\ge 1}\P^{\xi}(\rho_k)\log \P^{\xi}(\rho_k).
	\end{equation*}
\end{defn}

We now state Kaimanovich's conditional entropy criterion, which is one of the main tools to determine whether a given $\mu$-boundary is equal to the Poisson boundary of the $\mu$-random walk on $G$.

\begin{thm}[{\cite[Theorem 2]{Kaimanovich1985conditional}, \cite[Theorem 4.6]{Kaimanovich2000}}] \label{thm: conditional entropy criterion - Kaimanovich}
	Let $G$ be a countable group, and let $\mu$ be a probability measure on $G$ with $H(\mu)<\infty$. Consider a $\mu$-boundary $\mathbf{X}=(X,\lambda)$ of $G$. Then $\mathbf{X}$ is the Poisson boundary of $(G,\mu)$ if and only if for $\lambda$-almost every $\xi \in X$ we have $\lim_{n\to \infty}\frac{H_{\xi}(\alpha_n)}{n}=0.$
	
\end{thm}
Here $\alpha_n$ is the partition associated with the value of the $\mu$-random walk at instant $n$ (Definition \ref{def: partition rw at time n}). 

Following \cite[Section 5.1]{Rohlin1967}, let us define the \emph{mean conditional entropy over the $\mu$-boundary $\mathbf{X}$} as
\begin{equation*}
	H_{\mathbf{X}}(\rho)\coloneqq \int_{X}H_{\xi}(\rho)\ d\lambda(\xi).
\end{equation*}

The mean conditional entropy with respect to the Poisson boundary of $(G,\mu)$ has been considered in the proof of the entropy criterion for the triviality of the Poisson boundary in \cite{KaimanovcihVershik1983} (see Equation (12) at the bottom of page 464). For arbitrary $\mu$-boundaries, the mean conditional entropy is considered in \cite{Kaimanovich1985conditional} in the proof of the conditional entropy criterion (see the paragraph preceding Theorem 2 in page 194).

Another way of stating Kaimanovich's conditional entropy criterion (Theorem \ref{thm: conditional entropy criterion - Kaimanovich}) is the following, which uses the mean conditional entropy.

\begin{thm}\label{thm: conditional entropy alternative formulation}
	Let $G$ be a countable group, and let $\mu$ be a probability measure on $G$ with $H(\mu)<\infty$. Consider a $\mu$-boundary $\mathbf{X}=(X,\lambda)$ of $G$. Then $\mathbf{X}$ is the Poisson boundary of $(G,\mu)$ if and only if $
		\lim_{n\to \infty}\frac{H_{\mathbf{X}}(\alpha_n)}{n}=0.$
\end{thm}\
The mean conditional entropy can be obtained by taking the logarithm of the transition probabilities of the conditional random walk, and then integrating over the path space (see Equation (20) on page 462 of \cite{KaimanovcihVershik1983}). Thus, the formulation of the conditional entropy criterion in Theorem \ref{thm: conditional entropy alternative formulation} follows from Theorem \ref{thm: conditional entropy criterion - Kaimanovich}. We also mention that Theorem \ref{thm: conditional entropy alternative formulation} is a particular case of \cite[Theorem 2.17]{KaimanovichSobieczky2012}, where the result is proved in a more general setting for random walks along classes of graphed equivalence relations. For random walks on groups, the equivalence relation that one considers is the one induced by identifying sample paths of $G^{\mathbb{N}}$ which are mapped to the same boundary point $\xi \in X$. We thank Anna Erschler and Vadim Kaimanovich for pointing out to us the paper \cite{Kaimanovich1985conditional}, and for many explanations about the results that appear on it.

In the proofs of our main results (Theorems \ref{thm: lamplighter Poisson boundary} and \ref{thm: Poisson for finite first moment}) we will use the formulation of Kaimanovich's conditional entropy criterion in terms of the mean conditional entropy (Theorem \ref{thm: conditional entropy alternative formulation}). 

\subsection{Properties of entropy}
It will be important for our proofs that the mean conditional entropy of a partition over a $\mu$-boundary is at most the entropy of the partition. We state this in the following lemma.
\begin{lem}\label{lem: conditional entropy is at most the usual entropy}
	Consider a countable partition $\rho$ of the space of sample paths $G^{\mathbb{N}}$. Let $\mathbf{X}$ be a $\mu$-boundary of $G$ and $\mathbf{Y}$ a quotient of $\mathbf{X}$ with respect to a $G$-equivariant measurable partition. Then it holds that $H_{\mathbf{X}}(\rho) \leq H_{\mathbf{Y}}(\rho)$.
	
	In particular for every countable partition $\rho$ and each $\mu$-boundary $\mathbf{X}$, we have that $H_{\mathbf{X}}(\rho) \leq H(\rho)$.
\end{lem}
This follows from the fact that the function $x\mapsto -x\log(x)$ is concave on $[0,1]$, together with Jensen's inequality and the fact that conditional entropies are non-negative (see \cite[Section 5.10]{Rohlin1967}). 

We finish this section by introducing the notion of conditional entropy with respect to a partition. Consider two countable partitions $\rho,\gamma$ of the path space $G^{\mathbb{N}}$. Let $(X,\lambda)$ be a $\mu$-boundary of $G$. For every $\xi \in X$, denote 
\begin{equation*}
	\P^{\xi}(P\mid Q)\coloneqq \frac{\P^{\xi}(P\cap Q)}{P^{\xi}(Q)}, \text{ for }P\in \rho\text{ and }Q\in \gamma \text{ with }P^{\xi}(Q)\neq 0.
\end{equation*}

We define the associated \emph{conditional entropy of $\rho$ with respect to $\gamma$ conditioned on $\xi$} as
\begin{equation*}
	H_{\xi}(\rho\mid \gamma)\coloneqq -\sum_{P\in \rho, Q\in \gamma}\P^{\xi}(P\cap Q)\log \P^{\xi}(P\mid Q).
\end{equation*}
Define also 
\begin{equation}\label{eq: mean conditional entropy on boundary and partition}
	H_{\mathbf{X}}(\rho\mid \gamma)\coloneqq \int_{X}H_{\xi}(\rho\mid \gamma)\ d\lambda (\xi).
\end{equation}
We have the equality $H_{\mathbf{X}}(\rho\mid \gamma)=H_{\mathbf{X}}(\rho\vee \gamma)-H_{\mathbf{X}}(\gamma)$, where $\rho \vee \gamma \coloneqq \{P\cap Q\mid P\in \rho, Q\in \gamma \}$ is the \emph{join} of the partitions $\rho$ and $\gamma$.

\begin{rem}
	Rokhlin's theory of measurable partitions \cite{Rohlin1967} implies that for every $\mu$-boundary $\mathbf{X}$ of $G$, there is an associated measurable partition $\eta$ of the space of sample paths $G^{\mathbb{N}}$. With this, one can express the mean conditional entropy of a countable partition $\rho$
	over the $\mu$-boundary $\mathbf{X}$ as $H_{\mathbf{X}}(\rho)=H(\rho \mid \eta)$. Similarly, for two countable partitions $\rho, \gamma$ of the space of sample paths, we have $H_{\mathbf{X}}(\rho\mid \gamma)=H(\rho \mid \gamma \vee \eta)$. This is the notation used in \cite{KaimanovcihVershik1983,Kaimanovich1985conditional} when the mean conditional entropy is discussed.
	
	Throughout this paper we will use the notation $H_{\mathbf{X}}(\rho\mid \gamma)$ for the mean conditional entropy, in order to emphasize the dependence on the $\mu$-boundary that is being considered.
\end{rem}

In the proofs of our results we will use the following basic facts about entropy.
\begin{lem}\label{lem: basic properties entropy} Consider countable partitions $\rho,\gamma$ and $\delta$ of $G^{\mathbb{N}}$, and a $\mu$-boundary $\mathbf{X}$. The following properties hold.
	\begin{enumerate}
		\item $	H_{\mathbf{X}}(\rho\vee \gamma \mid \delta)= H_{\mathbf{X}}(\rho\mid \gamma\vee \delta)+H_{\mathbf{X}}(\gamma\mid\delta).$
		\item $H_{\mathbf{X}}(\rho\mid \gamma)\le H_{\mathbf{X}}(\rho\vee \delta\mid \gamma).$
		\item $H_{\mathbf{X}}(\rho\mid \gamma \vee \delta)\le H_{\mathbf{X}}(\rho\mid \gamma).$
	\end{enumerate}
\end{lem}
We refer to \cite[Corollaries 2.5 and 2.6]{MartinEngland1981} and \cite[Section 5]{Rohlin1967} for the proofs of these properties, and more generally to \cite[Chapter 2]{MartinEngland1981} for more details about entropy.
\begin{rem}\label{rem: partitions defined by random variables}
In order to simplify notation, throughout this paper we will use the same symbol to denote both a random variable and the partition of the space of sample paths that it defines. More precisely, let $X:(G^{\mathbb{N}},\P)\to D$ be a random variable from the space of sample paths with values on a countable set $D$. Then the countable partition $\rho_X$ of $(G^{\mathbb{N}},\P)$ defined by $X$ is given by saying that two sample paths  $\mathbf{w},\mathbf{w^\prime}\in G^{\mathbb{N}}$ belong to the same element of $\rho_X$ if and only if $X(\mathbf{w})=X(\mathbf{w^{\prime}})$. Then, we will denote $H(X)\coloneqq H(\rho_X)$.
\end{rem}

\section{Setup and the coarse trajectory}\label{section: the coarse trajectory}
\subsection{Definition of the boundary}\label{subsection: construction of the boundary}

Let $A$ and $B$ be countable groups, and consider their wreath product $A\wr B=\bigoplus_{B}A\rtimes B$ endowed with a probability measure $\mu$. Let us denote the $\mu$-random walk on $A\wr B$ as $w_n\coloneqq (\varphi_n,X_n)=(f_1,Y_1)\cdots (f_n,Y_n), \ n\ge 1,$ where $g_i\coloneqq (f_i,Y_i)$, $i\ge 1$, is a sequence of independent increments with law $\mu$.

The $\mu$-random walk on $A\wr B$ naturally projects to a random walk on $B$. Namely, denote by $\pi_B:A\wr B\to B$ the canonical projection, and consider $\mu_B\coloneqq \pi_{*}\mu$. Then $\{X_n\}_{n\ge 0}$ is a $\mu_B$-random walk on $B$. Let us denote by $(\partial B,\nu_B)$ the Poisson boundary of the random walk $(B,\mu_B)$. Note that $(\partial B,\nu_B)$ is also an $(A\wr B ,\mu)$-boundary, and denote by $\mathbf{bnd}_{B}:(A\wr B)^{\mathbb{N}}\to \partial B$ the associated boundary map.

\begin{defn}\label{def: stabilization of lamps}
	We say that the lamp configurations of the $\mu$-random walk $\{(\varphi_n,X_n)\}_{n\ge 0}$ on $A\wr B$ \emph{stabilize almost surely} if for every $b\in B$, there exists $N\ge 1$ such that $\varphi_n(b)=\varphi_N(b)$ for every $n\ge N$. 
\end{defn}

The hypothesis that lamp configurations stabilize almost surely implies that there is a map $\mathcal{L}:(A\wr B)^{\mathbb{N}}\to A^B$ that assigns to $\P$-almost every trajectory $w_n=(\varphi_n,X_n)$, $n\ge 1$, the limit value $\varphi_{\infty}$ of the lamp configurations $\varphi_n$, $n\ge 1$. Let us denote $\nu_{\mathcal{L}}\coloneqq\mathcal{L}_{*}\P$ the associated hitting measure. Since the map $\mathcal{L}$ is shift-invariant, the measure $\nu_{\mathcal{L}}$ is $\mu$-stationary, and the space $(A^B,\nu_{\mathcal{L}})$ is an $(A\wr B,\mu )$-boundary. The map $
	(A\wr B)^{\mathbb{N}}\to A^B \times \partial B, 
	w\mapsto \left(\mathcal{L}(w),\mathrm{bnd}_B(w)\right),$ pushes forward $\P$ to a $\mu$-stationary measure $\nu$ on $A^B\times \partial B$. With this, $(A^B\times \partial B,\nu)$ is a $\mu$-boundary of $A\wr B$. Theorem \ref{thm: lamplighter Poisson boundary} claims that in this setting, this space is the Poisson boundary of the $\mu$-random walk on $A\wr B$.
\subsection{Transience}
The following (not difficult) lemma states that the stabilization of lamp configurations implies that the projected random walk on $B$ is transient, as long as the step distribution has a non-trivial projection to the lamps group $\bigoplus_{B} A$. This observation goes back to the original paper of Kaimanovich-Vershik \cite{KaimanovcihVershik1983}, and we provide its proof for the convenience of the reader.
\begin{lem}\label{lem: the rw on B is transient}
	Let $\mu$ be a probability measure on $A\wr B$ such that lamp configurations stabilize almost surely. Assume that $\supp{\mu}$ is not completely contained in $B$. Then $\mu$ induces a transient random walk on $B$.
\end{lem}
\begin{proof}
	Indeed, suppose that the projection of the random walk to $B$ is recurrent. Consider an element $g=(f,x)\in A\wr B$ with $\mu(g)>0$, such that $\supp{f}\neq \varnothing$. Choose $b\in B$ such that $f(b)\neq e_A$. Then the Borel-Cantelli Lemma \cite[Lemma 2 of Section VIII.3]{Feller1968} implies that the lamp configuration at $b$ will change infinitely many times along almost every trajectory of the $\mu$-random walk.
\end{proof}

\begin{exmp} For probability measures on $A\wr B$ with an infinite first moment, the lamp configurations may fail to stabilize almost surely, even when the induced random walk on $B$ is transient. To illustrate this, let us consider $A=\Z/2\Z$ and $B=\Z^d$ with $d\ge 3$. Choose a sequence $\{\alpha_i\}_{i \ge 1}$ of positive numbers such that $\sum_{i\ge 1}\alpha_i=1$. Let $\mu_{\mathrm{srw}}$ be the uniform probability measure on the canonical symmetric generating set of $\Z^d$. For each $i\ge 1$ denote by $\eta_i$ the delta measure concentrated on the element of $\bigoplus_{\mathbb{Z}^d}\Z/2\Z$ that assigns the non-identity value $1$ to $b\in \Z^d$ if and only if $|b|\le i$, where $|\cdot|$ denotes the canonical word length on $\Z^d$.

Consider the probability measure $\frac{1}{2}\mu_{\mathrm{srw}}+\frac{1}{2}\sum_{i\ge 1}\alpha_i\eta_i$, which has a transient projection to $\Z^d$ since we assume $d\ge 3$. The measure has finite Shannon entropy if $-\sum_{i\ge 1}\alpha_i\log(\alpha_i)<\infty$. The event where the lamp configurations do not stabilize has probability at least as large as that of the event where for infinitely many values of $n$ there is a multiplication at time $n$ by an element in the support of $\eta_i$ for some $i\ge n$. The latter event has probability $1$ if $\sum_{i\ge 1}i\alpha_i$ diverges, due to the Borel-Cantelli lemma.

The above example should not lead one to believe that all heavy-tailed measures fail to exhibit the property of stabilization of lamp configurations. For every $i\ge 1$ denote by $\theta_i$ the uniform probability measure on the set of elements of $\bigoplus_{\mathbb{Z}^d}\Z/2\Z$ that assign $1$ to a unique $b\in \Z^d$ with $|b|=2^{i}$ and $0$ to all other elements, and consider the probability measure $\frac{1}{2}\mu_{\mathrm{srw}}+\frac{1}{2}\sum_{i\ge 1}\alpha_i\theta_i$. This measure has finite Shannon entropy if $\sum_{i\ge 1}\alpha_i\left(i(d-1)\log2-\log\alpha_i\right)<\infty$, and an infinite first moment when $\sum_{i\ge 1}\alpha_i 2^{i}$ diverges. By using the Borel-Cantelli lemma and the fact that the return probability to the identity at time $n$ on $\Z^d$ decays asymptotically as $n^{-d/2}$, one can see that the lamp configurations stabilize along almost every trajectory if $\sum_{i\ge 0}\alpha_i 2^{-i(\frac{d}{2}-1)}<\infty$, which always holds when $d\ge 3$.
\end{exmp}

Note that if the induced random walk on $B$ is recurrent, then Lemma \ref{lem: the rw on B is transient} implies that the support of $\mu$ must be completely contained in $B$. In particular, the space $(A^B,\nu_{\mathcal{L}})$ is trivial, since no modification to the lamp coordinate is possible. In that case, $(A^B\times \partial B,\nu)$ is isomorphic to $(\partial B,\nu_B)$, and the conclusion of Theorem \ref{thm: lamplighter Poisson boundary} follows. 

\subsection{Main Assumptions}\label{subsection: main assumptions} From now on and throughout the rest of the paper, we will assume that $H(\mu)<\infty$, that lamp configurations stabilize almost surely, and that the induced random walk on $B$ is transient.

We recall that in the entropy estimates below and in the following sections, we will use the same notation both for a random variable as well as for the partition of the space of sample path it defines (see Remark \ref{rem: partitions defined by random variables}).
\subsection{The coarse trajectory in the base group}
The following corresponds to Step (\labelcref{step: step 1}) of the sketch of the proof (Subsection \ref{subsection: sketch of the proof}).

\begin{defn}[Coarse trajectory]\label{def: coarse trajectory in base group}  Let $t_0\ge 1$ and $n\ge 1$. We define the \emph{$t_0$-coarse trajectory up to instant $n$} on the base group $B$ of the $\mu$-random walk by
	\begin{equation*}
		\mathcal{P}_n^{t_0}\coloneqq \left(X_{t_0}, X_{2t_0},\ldots, X_{\lfloor n/t_0\rfloor t_0 }\right)\in B^{\lfloor n/t_0\rfloor}.
	\end{equation*}
	That is, $\mathcal{P}_n^{t_0}$ is the ordered tuple of length $\lfloor n/t_0\rfloor$ formed by the random variables $X_{it_0}$, $i=1,\ldots, \lfloor n/t_0\rfloor$, that correspond to the projections of the $\mu$-random walk to $B$ every $t_0$ steps. When $n$ is clear from the context, we will refer to $\mathcal{P}_n^{t_0}$ simply as the $t_0$-coarse trajectory.
\end{defn}


\begin{lem}\label{lem: coarse traj has small entropy} Let $\varepsilon>0$. Then there exists $T\ge1$ such that for every $t_0\ge T$,
	\begin{equation*}\label{eq: coarse trajectory entropy}
		H_{\partial B}(\mathcal{P}_n^{t_0} )<\varepsilon n, \text{ for any }n\ge 1.
	\end{equation*}
\end{lem}

\begin{proof}
	Let $\varepsilon>0$. Then the conditional entropy criterion (Theorem \ref{thm: conditional entropy alternative formulation}) implies that there exists $T\ge 1$ such that for every $t_0\ge T$, we have $H_{\partial B}(X_{t_0})<\varepsilon t_0$.
	
	Let $n\ge 1$, $t_0\ge T$, and denote $s=\lfloor n/t_0\rfloor.$ Using Item 1 of Lemma \ref{lem: basic properties entropy}, we have that 
	\[ H_{\partial B}(X_{t_0},X_{2t_0},X_{3t_0},\ldots, X_{st_0})= H_{\partial B}(X_{t_0})+ H_{\partial B}(X_{2t_0},X_{3t_0},\ldots, X_{st_0}\mid X_{t_0}). \]
	
	Additionally, using Items 1 and 3 of Lemma \ref{lem: basic properties entropy} we have
	\begin{align*}
		H_{\partial B}(X_{2t_0},X_{3t_0},\ldots, X_{st_0}\mid X_{t_0})&= H_{\partial B}(X_{2t_0}\mid X_{t_0}) + H_{\partial B}(X_{3t_0},\ldots, X_{st_0}\mid X_{t_0},X_{2t_0})\\
		&\le H_{\partial B}(X_{2t_0}\mid X_{t_0}) + H_{\partial B}(X_{3t_0},\ldots, X_{st_0}\mid X_{2t_0}),
	\end{align*} 
	so that \[ H_{\partial B}(X_{t_0},X_{2t_0},X_{3t_0},\ldots, X_{st_0})\le H_{\partial B}(X_{t_0})+ H_{\partial B}(X_{2t_0}\mid X_{t_0}) + H_{\partial B}(X_{3t_0},\ldots, X_{st_0}\mid X_{2t_0}). \]
	
	By repeating this argument $s-1$ times, we obtain \[H_{\partial B}(X_{t_0},X_{2t_0},\ldots, X_{st_0})\le\sum_{j=0}^{s-1}H_{\partial B}\left(X_{(j+1)t_0}\mid X_{jt_0}\right).\] Finally, using the above together with the fact that $H_{\partial B}\left(X_{(j+1)t_0}\mid X_{jt_0}\right)=H_{\partial B}(X_{t_0})$ for every $j=0,\ldots,s-1$, we get
	\begin{equation*}
		H_{\partial B}(X_{t_0},X_{2t_0},\ldots, X_{st_0})\le sH_{\partial B}(X_{t_0})
		< s\varepsilon t_0\le \varepsilon n.
	\end{equation*}
\end{proof}
\subsection{Good and bad intervals}
We now divide the time interval between $0$ and $n$ into subintervals of length $t_0$. Afterward, we define what it means for these intervals to be \emph{good} or \emph{bad}, and introduce the partitions of bad increments. This corresponds to steps  \prange{step: step 3}{step: step 5} of the sketch of the proof (Subsection \ref{subsection: sketch of the proof}).
\begin{defn}\label{def: time intervals}
	Consider $t_0\ge 1$. For $j=1,2,\ldots,\lfloor n/t_0\rfloor$, let us define the \emph{$j$-th interval} $I_j\coloneqq \{(j-1)t_0+1,\ldots, jt_0\}$, and the \emph{final interval} $I_{\mathrm{final}}\coloneqq \{\lfloor n/t_0\rfloor t_0+1,\ldots, n \}$.
\end{defn}
Note that if $n$ is a multiple of $t_0$, then $I_{\mathrm{final}}=\varnothing$.

Recall that we denote by $\{g_i\}_{i\ge 1}$ the sequence of (independent, identically distributed) increments of the $\mu$-random walk on $A\wr B$.
\begin{defn}\label{def: bad elements and bad intervals}
	Let us consider finite subsets $R\subseteq B$ and $L\subseteq A$, with $e_B\in R$ and $e_A\in L$.
	\begin{itemize}
		\item Say that a group element $g=(f,x)$ is \emph{$(R,L)$-good} if $x\in R$, $\supp{f}\subseteq R$ and $f(b)\in L$ for every $b\in B$. Otherwise, call $g$ an \emph{$(R,L)$-bad} element.
		\item For each $j=1,2,\ldots,\lfloor n/t_0\rfloor$, say that the interval $I_j$ is \emph{$(R,L)$-good} if all the increments $g_i$, $i\in I_j$, are $(R,L)$-good. Otherwise, let us say that the interval $I_j$ is \emph{$(R,L)$-bad}.
	\end{itemize}
\end{defn}

Let us fix throughout the paper the auxiliary symbol $\star$. This symbol will play the role of unknown information, and it will serve as a tool to obtain low entropy events from high entropy ones. This will be carried out using the following general lemma about entropy, which is stated and proved \cite[Lemma 2.4]{ChawlaForghaniFrischTiozzo2022}.
\begin{lem}[Obscuring lemma]\label{lem: obscuring values of random variable}
	Let $D$ be a countable set, and fix an element $\star\notin D$. Let $X:(\Omega,\P)\to D$ be a random variable with $H(X)<\infty$. Then, for every $\varepsilon>0$, there exists $\delta>0$ such that for any measurable set $E\subseteq \Omega$ with $\P(E)<\delta,$ the random variable
	\begin{equation*}\label{eq: obscured random variable}
		\widetilde{X}(\omega)\coloneqq\begin{cases}
			X(\omega)&\text{ if }\omega\in E, \text{ and}\\
			\star&\text{ if }\omega\notin E,
		\end{cases}
	\end{equation*}
	satisfies $H(\widetilde{X})<\varepsilon$.
\end{lem}
\begin{defn}\label{defn: define the partition of bad and final intervals.}
	Let $t_0\ge 1$, and consider finite subsets $R\subseteq B$ and $L\subseteq A$ with $e_B\in R$ and $e_A\in L$.
	For every $j=1,2,\ldots,\lfloor n/t_0 \rfloor$, define the random variable
	\begin{equation*}
		Z_j\coloneqq\begin{cases}
			(g_i)_{i\in I_j}, &\text{ if }I_j \text{ is an }(R,L)\text{-bad interval, and}\\
			\star, &\text{otherwise,}
		\end{cases}
	\end{equation*}
	which has values in the space $(A\wr B)^{I_j}\cup \{\star\}$.
	Let us also define $Z_{\mathrm{final}}\coloneqq (g_i)_{i\in I_{\mathrm{final}}}$, which has values in the space $(A\wr B)^{I_{\mathrm{final}}}$. We define the \emph{$(t_0,R,L)$-bad increments} by
	\begin{equation*}
		\beta_n(t_0,R,L)\coloneqq (Z_{1},Z_{2},\ldots,Z_{\lfloor n/t_0 \rfloor},Z_{\mathrm{final}}).
	\end{equation*}
	
\end{defn}
\begin{lem}\label{lem: bad increments have small entropy} Let $t_0\ge 1$. For any $\varepsilon>0$, there exist finite subsets $R\subseteq B$ and $L\subseteq A$ with $e_B\in R$ and $e_A\in L$, and $K_1\ge 0$ such that for every $n\ge 1$,
	\begin{equation*}\label{eq: bad increments entropy}
		H(\beta_n(t_0,R,L))<\varepsilon n + K_1.
	\end{equation*}
\end{lem}
\begin{proof}
	Consider $t_0\ge 1$ and $\varepsilon>0$. First let us note that $I_{\mathrm{final}}$ consists of at most $t_0$ instants, and hence there is a constant $K_1\ge 0$, that only depends on $\mu$ and $t_0$, such that $H(Z_{\mathrm{final}})\le K_1$.
	
	Let us introduce the random variables $W_{j}$, $1\le j \le \lfloor n/t_0 \rfloor$, defined by $W_{j}=(g_i)_{i\in I_j}$. That is, $W_j$ corresponds to the increments of the $\mu$-random walk that occurred during the interval $I_j$. Note that the variables $W_{j}$ are identically distributed and have finite entropy. By using Lemma \ref{lem: obscuring values of random variable}, we can find $\delta>0$ such that if $E\subseteq (A\wr B)^{t_0}$ satisfies $\mu^{t_0}(E)<\delta$, then the variable
	\begin{equation*}
		\widetilde{W}_{j}=\begin{cases}
			(g_i)_{i\in I_j}, &\text{ if }E\text{ occurs}, \text{ and}\\
			\star, &\text{ otherwise,}
		\end{cases}
	\end{equation*}
	satisfies $H(\widetilde{W}_{j})<\varepsilon$.
	
	Choose finite subsets $R\subseteq B$ and $L\subseteq A$ with $e_B\in R$ and $e_A\in L$, such that
	\begin{equation*}
		\mu^{t_0}(I_1 \text{ is an }(R,L)\text{-bad interval})<\delta.
	\end{equation*}
	For every $1\le j \le \lfloor n/t_0 \rfloor$, consider the random variable $\widetilde{W}_{j}$, associated with the event $E_j\coloneqq \left \{I_j \text{ is an }(R,L)\text{-bad interval} \right\}$. Then $H(\widetilde{W}_{j})<\varepsilon$, and we have $Z_j=\widetilde{W}_{j}$. With this, we conclude that $
		H(\beta_n(t_0,R,L))\le \sum_{j=1}^{ \lfloor n/t_0 \rfloor}H(Z_j)+Z_{\mathrm{final}}<\varepsilon n +K_1,	\ \text{ for every }n\ge 1.$
	
\end{proof}

The following lemma says that the trajectory at time $n$ can be deduced from the trajectory at the last time instant that is a multiple of $t_0$, together with the partition of $(t_0,R,L)$-bad intervals. The statement is expressed in terms of the mean conditional entropy (see Equation \eqref{eq: mean conditional entropy on boundary and partition}) with respect to the boundary of infinite lamp configurations $A^B$.

\begin{lem}\label{lem: it suffices to guess the last multiple of t0} Let $n\ge 1$, $t_0\ge 1$ and consider finite subsets $R\subseteq B$ and $L\subseteq A$ with $e_B\in R$ and $e_A\in L$. For every $n\ge 1$,
	\begin{equation*}
		H_{A^B}\left(\varphi_n,X_n\mid \varphi_{s}\vee\mathcal{P}_n^{t_0}\vee\beta_n(t_0,R,L)\right)=0,	\text{ where we denote } s=\lfloor n/t_0\rfloor t_0.
	\end{equation*}
\end{lem}

\begin{proof}
	The position $(\varphi_s,X_s)$ of the $\mu$-random walk at time $s$ is completely determined by $\varphi_s$ and $\mathcal{P}_n^{t_0}$, which contains the value of $X_s$. Furthermore, $\beta_n(t_0,R,L)$ contains the increments $Z_{\mathrm{final}}$ of the last interval. With this information, we can completely determine $(\varphi_n,X_n)$.
\end{proof}

\subsection{The coarse neighborhood}
Now we introduce the coarse neighborhood of the trajectory and do the entropy estimates of Step (\labelcref{step: step 6}) of the sketch of the proof (Subsection \ref{subsection: sketch of the proof})
\begin{defn}\label{def:neighborhood of coarse traj}
	Let $t_0\ge 1$ and $R\subseteq B$ be a finite subset with $e_B\in R$. Define the \emph{$(t_0,R)$-coarse neighborhood} of the trajectory at instant $n$ by
	\begin{equation*}
		\mathcal{N}_n(t_0,R)\coloneqq \bigcup_{j=0}^{\lfloor n/t_0 \rfloor -1} \left\{ X_{jt_0}r_1r_2\cdots r_{t_0}\mid r_k\in R, \text{ for }k=1,2,\ldots,t_0 \right \}.
	\end{equation*}
\end{defn}

In order to estimate the entropy of the $n$-th instant of the random walk, it will be useful to divide the values of the lamp configuration into the ones inside the coarse neighborhood and the ones outside of it.

\begin{defn}\label{def: entropy in and out lamps} Consider $t_0\ge 1$ and a finite subset $R\subseteq B$ with $e_B\in R$. We define the \emph{lamp configuration inside the $(t_0,R)$-coarse neighborhood} $\mathcal{N}_n(t_0,R)$ as
	\begin{equation*} \Phi_n^{\mathrm{in}}(t_0,R)\coloneqq \varphi_n|_{\mathcal{N}_n(t_0,R)},
	\end{equation*}
	and the \emph{lamp configuration outside the $(t_0,R)$-coarse neighborhood} $\mathcal{N}_n(t_0,R)$ as
	\begin{equation*}
		\Phi_{n}^{\mathrm{out}}(t_0,R)\coloneqq \varphi_n|_{B\backslash \mathcal{N}_n(t_0,R)}.
	\end{equation*}
\end{defn}
\begin{lem}\label{lem: lamp config outside coarse neighborhood has small entropy}
	Let $t_0\ge 1$, and let $R\subseteq B$, $L\subseteq A$ be finite subsets with $e_B\in R$, $e_A\in L$. Then for any $n\ge 1$ we have $
		H\left(\Phi_{s}^{\mathrm{out}}(t_0,R)\Big| \mathcal{P}_n^{t_0}\vee \beta_n(t_0,R,L) \right)=0,$ where $s=\lfloor n/t_0\rfloor t_0$. 
\end{lem}
\begin{proof}
	Consider $t_0,R$ and $L$ as above. Note that any modification of the lamp configuration $\varphi_s$ at an element $b\notin \mathcal{N}_n(t_0,R)$ must have occurred during an $(R,L)$-bad interval. Since we condition on the information of the $(t_0,R,L)$-bad increments and $\mathcal{P}_n^{t_0}$, we know the exact value of every $(R,L)$-bad increment and the position on $B$ of the $\mu$-random walk at the moment that it was applied. This completely determines the value of $\varphi_s$ in every element of $B\backslash\mathcal{N}_n(t_0,R)$.
\end{proof}

\section{Entropy inside the coarse neighborhood}\label{section: entropy inside the coarse neighborhood}
The objective of this section is to prove the following proposition, which  corresponds to Step (\labelcref{step: step 7}) of the sketch of the proof (Subsection \ref{subsection: sketch of the proof}).
\begin{prop}\label{prop: the lamp config INSIDE the neighborhood has small entropy}
	For any $\varepsilon>0$ there exists $T\ge 1$ such that the following holds. Consider any finite subsets $R\subseteq B$ and $L\subseteq A$ with $e_B\in R$ and $e_A\in L$, and $t_0\ge T$. Then there exist $N\ge 1$ and $K\ge 0$ such that
	\begin{equation}\label{eq: lamp config INSIDE}
		H_{A^{B}}\left( \Phi_{s}^{\mathrm{in}}(t_0,R)\Big| \mathcal{P}_n^{t_0}\vee\beta_n(t_0,R,L)\right)<\varepsilon n + K,
	\end{equation}
	for every $n\ge N$, where we denote $s=\lfloor n/t_0\rfloor t_0$.
\end{prop}

This result is proved in Subsection \ref{subsection: proof of LAMPS INSIDE}, for which we first prove some additional lemmas in Subsections \ref{subsection: unstable points in the base group}, \ref{subsection: visits to unstable points} and \ref{subsection: lamp increments at unstable elements}. In the case where the lamp group $A$ is finite, Proposition \ref{prop: the lamp config INSIDE the neighborhood has small entropy} can be proved directly, without passing through intermediate lemmas. This proof is presented in Subsection \ref{subsection: proof of lamps inside for FINITE lamp group} to expose the main ideas that go into the proof of Proposition \ref{prop: the lamp config INSIDE the neighborhood has small entropy}, and does not play a role in the proof of the general case where the lamp group $A$ may be infinite in Subsection \ref{subsection: proof of LAMPS INSIDE}.

\subsection{Proof of Proposition \ref{prop: the lamp config INSIDE the neighborhood has small entropy} when the lamp group $A$ is finite}\label{subsection: proof of lamps inside for FINITE lamp group}
\begin{proof}
	Let $\varepsilon>0$, and let us consider $T=1$. Consider any finite subsets $R\subseteq B$ and $L\subseteq A$ with $e_B\in R$ and $e_A\in L$, and let $t_0\ge T$. 
	
	Recall that we assume that the lamp configuration $\varphi_n$, $n\ge 1$, of the $\mu$-random walk on $A\wr B$ stabilizes almost surely to an infinite lamp configuration $\varphi_{\infty}$. Let us define the event $ E_{i,n}\coloneqq \left\{\text{there exists }b\in X_{it_0}R^{t_0}\text{ such that } \varphi_{n}(b)\neq \varphi_{\infty}(b)\right \}$ for each $n\ge 1$ and $i\in \left\{0,1,\ldots, \lfloor \frac{n}{t_0} \rfloor-1\right\}$. 
		Let us fix $\delta>0$ that satisfies 
	\begin{equation}\label{eq: choice of delta}
		\delta<\frac{\varepsilon}{2|R^{t_0}|\log(|A|)}, \text{ and } -\delta\log(\delta)-(1-\delta)\log(1-\delta)<\varepsilon/2.
	\end{equation} The hypothesis of stabilization along sample paths implies that there exists $N\ge 1$ such that for all $n\ge N$ we have $\P\left(E_{0,s}\right)<\delta$, where $s=\lfloor n/t_0\rfloor t_0$. Note that thanks to the Markov property and the group invariance of the process, for all $i\in \left\{0,1,\ldots, \lfloor \frac{n}{t_0} \rfloor-1\right\}$ we have $\P\left(E_{i, s}\right)=\P\left(E_{0,s-it_0}\right).$ Together with the above, for each $n\ge N$ and $i=0,1,\ldots, \left\lfloor\frac{n-N+1}{t_0}\right\rfloor $ we have that $\P\left(E_{i,s}\right)<\delta.$

	Let us now prove that $H_{A^{B}}\left( \Phi_{s}^{\mathrm{in}}(t_0,R)\Big| \mathcal{P}_n^{t_0}\vee \beta_n(t_0,R,L)\right)<\varepsilon n + K$ for all $n\ge N$, where $s=\lfloor n/t_0\rfloor t_0$. In order to do so, it suffices to prove that for all $n\ge N$ and $i=0,1,\ldots, \left\lfloor\frac{n-N+1}{t_0}\right\rfloor$ we have
	
	\begin{equation}\label{eq: goal of lamps inside for FINITE LAMPS GROUP}
		H_{A^B}\left( \varphi_{s}|_{N_i} \mid \mathcal{P}_n^{t_0}\vee\beta_n(t_0,R,L) \right)<\varepsilon, 
	\end{equation}
	where we denote $N_i\coloneqq X_{it_0}R^{t_0}$ the $R^{t_0}$-neighborhood of $X_{it_0}$. Indeed, $\varphi_{s}|_{\mathcal{N}_n(t_0,R)}$ is completely determined by the values $\varphi_{s}|_{N_i}$, with $i=0,1,\ldots, \left\lfloor\frac{n}{t_0}\right\rfloor$. Thanks to Equation \eqref{eq: goal of lamps inside for FINITE LAMPS GROUP}, the terms corresponding to $i=0,1,\ldots, \left\lfloor\frac{n-N+1}{t_0}\right\rfloor$ each contribute an amount of $\varepsilon$ to the entropy of $\varphi_{s}|_{\mathcal{N}_n(t_0,R)}$, whereas the remaining values $i= \left\lfloor\frac{n-N+1}{t_0}\right\rfloor+1, \ldots, \left\lfloor\frac{n}{t_0}\right\rfloor$ contribute at most $K\coloneqq H\left(\mu^{*(N+1+t_0)}\right)$.
	
	In order to finish the proof, let us show that Equation \eqref{eq: goal of lamps inside for FINITE LAMPS GROUP} holds. For each $i=0,1,\ldots, \left\lfloor\frac{n-N+1}{t_0}\right\rfloor$, denote by $\mathcal{E}_i=\{E_{i,s},E_{i,s}^c\}$ the partition of the space of sample paths induced by the event $E_{i,s}$ defined above. Note that we have that $H(\mathcal{E}_i)\le -\delta\log(\delta)-(1-\delta)\log(1-\delta)<\varepsilon/2$,  since $\P(E_{i,s})<\delta$ for $\delta$ as in Equation \eqref{eq: choice of delta}. 
	
In order to estimate the value of $H_{\varphi_{\infty}}\left( \varphi_{s}|_{N_i} \mid \mathcal{P}_n^{t_0}\vee\beta_n(t_0,R,L)\vee \mathcal{E}_i \right)$, we remark that in the event $E_{i,s}^c$, the lamp configuration at time $s$ coincides with the limit lamp configuration $\varphi_{\infty}$. This implies that for each $\varphi_{\infty}\in A^B$ we have
		\begin{align*}
				H_{\varphi_{\infty}}\left( \varphi_{s}|_{N_i} \mid \mathcal{P}_n^{t_0}\vee\beta_n(t_0,R,L)\vee \mathcal{E}_i \right)&\le \log\left(|A|^{|N_i|}\right)\P(E_{i,s}) \le \log\left(|A|^{|N_i|}\right)\delta.
		\end{align*}

	Now we integrate with respect to $\nu_{\mathcal{L}}$ over all $\varphi_{\infty}\in A^B$, and we get 
	\begin{equation*}
		H_{A^B}\left( \varphi_{s}|_{N_i} \mid \mathcal{P}_n^{t_0}\vee\beta_n(t_0,R,L)\vee \mathcal{E}_i \right)\le \log\left(|A|^{|N_i|}\right)\delta.
	\end{equation*}

	By noticing that $|N_i|=|R^{t_0}|$ and thanks to our choice of $\delta$ as in Equation \eqref{eq: choice of delta}, we see that $\log\left(|A|^{|N_i|}\right)\delta<\varepsilon/2$. From this, using the general properties of entropy from Lemma \ref{lem: basic properties entropy} together with what we just showed, we conclude that
	\begin{align*}
		H_{A^B}\left( \varphi_{s}|_{N_i} \mid \mathcal{P}_n^{t_0}\vee\beta_n(t_0,R,L) \right)&\le H_{A^B}\left( \varphi_{s}|_{N_i} \mid \mathcal{P}_n^{t_0}\vee\beta_n(t_0,R,L)\vee \mathcal{E}_i \right)+H(\mathcal{E}_i)\\
		&<\varepsilon/2 +\varepsilon/2=\varepsilon.
	\end{align*}
	This proves Equation \eqref{eq: goal of lamps inside for FINITE LAMPS GROUP}.
\end{proof}

In the following subsections we prove Proposition \ref{prop: the lamp config INSIDE the neighborhood has small entropy} in its full generality, where the lamp group $A$ may be infinite.
\subsection{Unstable points in the base group}\label{subsection: unstable points in the base group}

Recall that we denote by $\varphi_{\infty}$ the limit lamp configuration associated with a trajectory $\{(\varphi_n,X_n)\}_{n\ge 0}$ of the $\mu$-random walk on $A\wr B$. We now define our notion of \emph{unstable points} at a given time instant, which are the positions on $B$ for which the lamp configuration has not yet stabilized to its limit value.

\begin{defn}\label{def: unstable elements} Let $t_0\ge 1$ and consider finite subsets $R\subseteq B$ and $L\subseteq A$ with $e_B\in R$ and $e_A\in L$. For $n\ge 1$, let us define the set of \emph{unstable points at time $s\coloneqq\lfloor n/t_0\rfloor t_0$} as
	\begin{equation*}
		\mathcal{U}_s\coloneqq \left \{b\in \mathcal{N}_n(t_0,R)\mid \varphi_s(b)\neq\varphi_{\infty}(b) \right \}.
	\end{equation*}
	Note that $\mathcal{U}_s$ is a random subset of $B$ that depends on the entire trajectory of the $\mu$-random walk on $A\wr B$.
\end{defn}

The next lemma says that, in expectation, there are not too many unstable points.

\begin{lem}\label{lem: expectation of unstable points is small} Let $t_0\ge 1$ and consider finite subsets $R\subseteq B$ and $L\subseteq A$ with $e_B\in R$ and $e_A\in L$. For any $\varepsilon>0$, there exists a constant $K_2>0$ such that for every $n\ge 1$, 
	\begin{equation*}\E\left(|\mathcal{U}_s| \mid \mathcal{P}_n^{t_0}\right)<\varepsilon n+K_2, \text{ where we denote }s=\lfloor n/t_0\rfloor t_0.
	\end{equation*}
\end{lem}
\begin{proof}
	Consider $t_0,R$ and $L$ as above. For $0\le j\le \lfloor n/t_0\rfloor-1$, let us denote $F_j\coloneqq X_{jt_0}R^{t_0}$ and recall that $\mathcal{N}_n(t_0,R)=\bigcup_{j=0}^{\lfloor n/t_0\rfloor-1}F_j$ (Definition \ref{def:neighborhood of coarse traj}).
	
	For $n_0\ge 1$, let us say that an instant $j$ is \emph{poorly $n_0$-stabilized} if the lamp configuration at $F_j$ is modified at some instant beyond $jt_0+n_0$. That is, if for some $k>jt_0+n_0$ and some $b\in F_j$, we have $\varphi_{k-1}(b)\neq \varphi_{k}(b)$.
	
	We observe the following.\\
	\textbf{Claim 1:} \textit{We have that \( \P(j \text{ is poorly }n_0\text{-stabilized})=\P(0 \text{ is poorly }n_0\text{-stabilized})\) for every $n_0\ge 1$ and $j\ge 0$.}

	Indeed, we see that
	\begin{align*}
		\P(j \text{ is poorly }n_0\text{-stabilized})&=\P(\varphi_{k-1}(b)\neq \varphi_{k}(b)\text{ for some }k>jt_0+n_0\text{ and }b\in X_{jt_0}R^{t_0})\\
		&=\P(\varphi_{k-1}(b)\neq \varphi_{k}(b)\text{ for some }k>n_0\text{ and }b\in R^{t_0})\\
		&=\P(0 \text{ is poorly }n_0\text{-stabilized}).
	\end{align*}
	In the second equality above we used the Markov property and the group invariance of the random walk to change the time instants beyond $jt_0+n_0$ by those beyond $n_0$.
	
	\textbf{Claim 2:}\textit{ $\lim_{n_0\to \infty}\P(0 \text{ is poorly }n_0\text{-stabilized})=0$.}\\
	Indeed, this follows from the hypothesis that lamp configurations stabilize almost surely, which implies that $\P(\text{the lamp configuration in }R^{t_0} \text{ is modified after time $n_0$})\xrightarrow[n_0\to \infty]{}0.$
	
	Let $\varepsilon>0$. Using Claims 1 and 2 we can find $n_0\ge 1$ such that for every $k\ge n_0$ and $j\ge 0$ we have $
		\P(j \text{ is poorly }k\text{-stabilized})\le \P(j \text{ is poorly }n_0\text{-stabilized})<\varepsilon.$
	
	Thanks to the linearity of expectation, we have
	\begin{align*}
		\E\left(|\mathcal{U}_s| \mid \mathcal{P}_n^{t_0}\right)&= 
		\E\left(\sum_{j=0}^{\lfloor n/t_0\rfloor-1}\mathds{1}_{\left\{\varphi_s|_{F_j}\neq \varphi_{\infty}|_{F_j}\right \}} \mid \mathcal{P}_n^{t_0}\right)\\
		&=
		\sum_{j=0}^{\lfloor n/t_0\rfloor-1}\P\left(\varphi_s|_{F_j}\neq \varphi_{\infty}|_{F_j}\mid \mathcal{P}_n^{t_0}\right)=\sum_{j=0}^{\lfloor n/t_0\rfloor-1} \P(j\text{ is poorly }s\text{-stabilized}).
	\end{align*}
	
	Denote $J=\left \lfloor \lfloor n/t_0 \rfloor- n_0/t_0-1  \right \rfloor$, which is the largest integer that satisfies the inequality $Jt_0+t_0\le \lfloor n/t_0 \rfloor t_0 - n_0$. In other words, $Jt_0$ is the last instant whose associated interval $I_{Jt_0}$ still has $n_0$ remaining units of time before reaching time $s$. Then if $j\in \{0,1,\ldots, J\}$, it holds that $  \P(j\text{ is poorly }s\text{-stabilized})\le \P(0\text{ is poorly }n_0\text{-stabilized}).$ From this, we see that
	\begin{align*}
		\E\left(|\mathcal{U}_s| \mid \mathcal{P}_n^{t_0}\right)&\le\sum_{j=0}^{J} \P(j\text{ is poorly }s\text{-stabilized})+\sum_{j=J+1}^{\lfloor n/t_0\rfloor-1}\P(j\text{ is poorly }s\text{-stabilized})\\
		&\le (J+1) \P(0\text{ is poorly }n_0\text{-stabilized})+ \lfloor n/t_0\rfloor-1 -J \\
		&< (J+1)\varepsilon +\lfloor n/t_0\rfloor-1 -J.
	\end{align*}
	Our choice of $J$ guarantees that $J+1\le (n-n_0)/t_0\le n$ and $\lfloor n/t_0\rfloor-1 -J\le n_0/t_0+1$, so that $
		\E\left(|\mathcal{U}_s| \mid \mathcal{P}_n^{t_0}\right)\le \varepsilon n + K_2,$ where $K_2=n_0/t_0+1$ is a constant that only depends on $\mu$, $t_0$, $R$ and $\varepsilon$.
\end{proof}

Now we will prove that the small expectation of the number of unstable points implies that revealing their values is a low entropy event. For this, we will need a general lemma about entropy, whose proof is a consequence of basic inequalities in information theory. Since we were unable to find a reference for this result in the literature, we present the proof below for the convenience of the reader.
\begin{lem}\label{lem: low expectation implies low entropy} Let $D$ be a non-empty finite set. Let $Z$ be a random subset of $D$. That is, a random variable that takes values in the power set of $D$. Suppose that $\E(|Z|)\le \alpha |D|$ for some $0\le \alpha \le 1/2$. Then
	\begin{equation*}
		H(Z)\le |D|(-\alpha \log \alpha - (1-\alpha)\log (1-\alpha)).
	\end{equation*}
\end{lem}
\begin{proof}
	Let us define the binary entropy function $h(p)\coloneqq -p\log p -(1-p)\log (1-p)$, for $0\le p \le 1$. The value $h(p)$ coincides with the entropy of a Bernoulli random variable $X$ with values in $\{0,1\}$, and with $\P(X=1)=p$. It is well-known that for a given $0\le p\le 1$ and every random variable $Q$ that takes values in a two-element set $\{a,b\}$ with $\P(Q=a)=p$, it holds that $H(Q)\le h(p)$. Furthermore, the function $h(p)$ is concave, and it is strictly increasing in the interval $0\le p \le 1/2$.
	
	Let us now prove the lemma. For every $d\in D$, let us define $Z_d\coloneqq \mathds{1}_{ \{d\in Z\} }$. Note that $Z_d$ is a Bernoulli random variable with parameter $p_d\coloneqq \P(d\in Z)$. Then $Z$ is completely determined by $\{Z_d\}_{d\in D}$, and we have $\E(|Z|)=\sum_{d\in D}p_d\le \alpha |D|$, for a fixed $0\le \alpha \le 1/2$.
	
	First, we note that $
		H(Z)=H(\{Z_d \}_{d\in D})\le \sum_{d\in D}H(Z_d) \le \sum_{d\in D}h(p_d).$
	Now, using Jensen's inequality, since $h$ is concave and $\alpha\le 1/2$, we get 
	\begin{equation*}
		\sum_{d\in D}h(p_d)=|D|\sum_{d\in D}\frac{1}{|D|}h(p_d)\le |D|h\left(\sum_{d\in D}\frac{1}{|D|}p_d \right)\le |D|h\left(\frac{\E(|Z|)}{|D|} \right)\le|D|h(\alpha).
	\end{equation*}
\end{proof}

\begin{lem}[The set of unstable points has small entropy]\label{lem: unstable points have small entropy}
	Let $t_0\ge 1$ and consider finite subsets $R\subseteq B$ and $L\subseteq A$ with $e_B\in R$ and $e_A\in L$. For every $\varepsilon>0$, there exists $N_1\ge 1$ such that $H(\mathcal{U}_s\mid \mathcal{P}_n^{t_0})<\varepsilon n$ for any $n\ge N_1$, where we denote $s=\lfloor n/t_0\rfloor t_0$.
\end{lem}
\begin{proof}
	Let us consider $t_0, R$ and $L$ as above. Since $\mathcal{U}_s$ is a subset of $\mathcal{N}_n(t_0,R)$, we know that the size of $\mathcal{U}_s$ is at most $n|R|^{t_0}$.  Consider an arbitrary $\varepsilon >0$, and let us consider $\widetilde{\varepsilon}>0$ to be fixed later. Using Lemma \ref{lem: expectation of unstable points is small} we can find a constant $K_2\ge 0$ such that 
	\begin{equation*}
		\E\left(|\mathcal{U}_s| \mid \mathcal{P}_n^{t_0}\right)<\widetilde{\varepsilon}n+ K_2=\frac{t_0\widetilde{\varepsilon} +K_2\frac{t_0}{n}}{|R^{t_0}|} \frac{n}{t_0}|R^{t_0}|.
	\end{equation*}
	Now we are going to use Lemma \ref{lem: low expectation implies low entropy}. Let us define $M_n\coloneqq \frac{n}{t_0}|R^{t_0}|$ and let $D_n$ be an enumeration of the coarse trajectory $ \mathcal{U}_s $ with fixed symbols from $1$ up to $M_n$. This is a codification that gives rise to an injective map from subsets of the coarse neighborhood to subsets of indices in $D_n$. We then have that $H(\mathcal{U}_s\mid \mathcal{P}_n^{t_0})\le H(f(\mathcal{U}_s))$. We proved above that for $\alpha=\frac{t_0\widetilde{\varepsilon} +K_2\frac{t_0}{n}}{|R^{t_0}|}$, it holds that $\E(|f(\mathcal{U}_s)|)<\alpha |D_n|$. If $\widetilde{\varepsilon}$ is small enough and $n$ is large enough, then $\alpha<1/2$ and we can apply Lemma \ref{lem: low expectation implies low entropy} to show that
	\begin{equation*}
		H\left(\mathcal{U}_s\mid \mathcal{P}_n^{t_0}\right)\le  H(f(\mathcal{U}_s))\le  \frac{n}{t_0}|R|^{t_0} (-\alpha \log \alpha - (1-\alpha)\log (1-\alpha) ).
	\end{equation*}
	The statement of the lemma follows by choosing $\widetilde{\varepsilon}$ small enough and $N_1$ large enough such that $-\alpha \log \alpha - (1-\alpha)\log (1-\alpha) <\varepsilon t_0/|R^{t_0}|$.
\end{proof}

\subsection{Visits to unstable points}\label{subsection: visits to unstable points}

We now give estimates for the entropy associated with knowing the instants up to time $n$ where there could have been a lamp modification at an unstable point by a good element (in the sense of Definition \ref{def: bad elements and bad intervals}). Lemmas \ref{lem: expectation of unstable points is small}, \ref{lem: unstable points have small entropy} and \ref{lem: visit times to unstable points have small entropy} below correspond to Step (\labelcref{step: step 8}) of the sketch of the proof (Subsection \ref{subsection: sketch of the proof}).

\begin{defn}\label{def: visits to unstable points} Let $t_0\ge 1$ and consider finite subsets $R\subseteq B$ and $L\subseteq A$ with $e_B\in R$ and $e_A\in L$. Let $n\ge 1$ and denote $s=\lfloor n/t_0\rfloor t_0$. Let us define the times of \emph{visit to unstable points} as
	\begin{equation*}
		\mathcal{V}_s=\{j\in\{1,2,\ldots,\lfloor n/t_0\rfloor\} \mid I_j\text{ is an }(R,L)\text{-good interval and }X_{(j-1)t_0}R^{t_0}\cap \mathcal{U}_s\neq \varnothing\}.
	\end{equation*}
\end{defn}
\begin{lem}\label{lem: not many visits in expectation} Let $t_0\ge 1$ and consider finite subsets $R\subseteq B$ and $L\subseteq A$ with $e_B\in R$ and $e_A\in L$. For any $\varepsilon>0$, there exists $K_3\ge 1$ such that for every $n\ge 1$,
	\begin{equation*}
		\E\left(|\mathcal{V}_s|\mid \mathcal{P}_n^{t_0}\right)<\varepsilon n +K_3 \text{, where we denote }s=\lfloor n/t_0\rfloor t_0.
	\end{equation*}
\end{lem}
\begin{proof}
	Let $t_0$, $R$ and $L$ be as above.	For an arbitrary $\ell \ge 1$, define 
	\begin{align*}
		J_{\ell}\coloneqq \{j\in \mathbb{N}\mid \text{ there exist }n_1>n_2>\ldots> n_\ell>(j-1)t_0 &\text{ such that }\\
		&X_{n_k}=X_{(j-1)t_0}, \text{ for }k=1,2,\ldots,\ell \}.
	\end{align*}
	In other words, the set $J_{\ell}$ is formed by the values $j\ge 1$ such that the element $X_{(j-1)t_0}$ is visited at least $\ell$ times more in the future.
	
	\textbf{Claim 1:} For every $j\ge 1$, we have  $\P(1\in J_\ell)=\P(j\in J_\ell)$.
	
	Indeed, by using the Markov property, we see that for every $j\ge 1$,
	\begin{align*}
		\P(j\in J_\ell)&=\P(X_{(j-1)t_0} \text{ is visited more than }\ell \text{ times in the future})\\
		&=\P(e_B \text{ is visited more than }\ell \text{ times in the future})=\P(1\in J_{\ell}).
	\end{align*}
	
	\textbf{Claim 2:} We have $\lim_{\ell \to \infty}\P(1\in J_\ell)=0.$\\
	Indeed, recall that we assume that the induced random walk on $B$ is transient (see Subsection \ref{subsection: main assumptions} and the paragraph above it). The claim is then just a consequence of the transience of this random walk. \\
	
	\textbf{Claim 3:} $\E(|\mathcal{V}_s\backslash J_\ell|\mid \mathcal{P}_n^{t_0})\le \ell |R^{t_0}| \E(|\mathcal{U}_s|\mid \mathcal{P}_n^{t_0})$.\\
	Indeed, note that
	\begin{align*}
		\E\left( |\mathcal{V}_s\backslash J_\ell|\mid \mathcal{P}_n^{t_0} \right)&\le\E\left( \sum_{j=1}^{\lfloor n/t_0\rfloor}\mathds{1}_{\{X_{(j-1)t_0}R^{t_0}\cap \mathcal{U}_s\neq \varnothing\}}\mathds{1}_{\left\{\substack{ X_{(j-1)t_0}\text{ is visited }\le \ell \\ \text{ times after instant }jt_0} \right\} }\Big| \mathcal{P}_n^{t_0} \right)\\
		&=\E\left(\sum_{u\in \mathcal{U}_s}\sum_{r\in R^{t_0}} \sum_{j=1}^{\lfloor n/t_0\rfloor}\mathds{1}_{\{X_{(j-1)t_0}=ur^{-1}\}}\mathds{1}_{\left\{\substack{ X_{(j-1)t_0}\text{ is visited }\le \ell \\ \text{ times after instant }jt_0} \right\} }\Big| \mathcal{P}_n^{t_0} \right)\\
		&\le\E\left(|\mathcal{U}_s| |R^{t_0}|\ell\Big| \mathcal{P}_n^{t_0} \right)=\ell |R^{t_0}| \E(|\mathcal{U}_s|\mid \mathcal{P}_n^{t_0}).
	\end{align*}
	
	Consider $\varepsilon>0$ arbitrary. Thanks to Claim 2, we can find $\ell\ge 1$ large enough such that $p_{\ell}\coloneqq \P(0\in J_\ell )<\frac{\varepsilon}{2t_0}.$
	By virtue of Lemma \ref{lem: expectation of unstable points is small}, we can find $K_2\ge 0$ such that $		\E(|\mathcal{U}_s|\mid \mathcal{P}_n^{t_0})<\frac{\varepsilon}{2\ell|R^{t_0}|} n + K_2$. Now, using Claims 1 and 3, we see that
	\begin{align*}
		\E\left(|\mathcal{V}_s|\mid \mathcal{P}_n^{t_0}\right) &=\E\left(|J_{\ell}|\mid \mathcal{P}_n^{t_0}\right)+\E\left(|\mathcal{V}_s\backslash J_{\ell}| \mid \mathcal{P}_n^{t_0}\right)\\
		&\le \sum_{j=1}^{t_0}\P(j\in J_{\ell}) + \ell|R^{t_0}| \E(|\mathcal{U}_s|\mid \mathcal{P}_n^{t_0})\\
		&< t_0p_{\ell}+ \ell|R^{t_0}|\left(\frac{\varepsilon}{2\ell|R^{t_0}|} n + K_2\right)<\varepsilon n +\ell |R^{t_0}| K_2.
	\end{align*}
	The statement of the lemma follows from setting $K_3\coloneqq K_2\ell |R^{t_0}|$, which is a constant that depends only on $\mu$, $t_0$, $R$ and $\varepsilon$.
\end{proof}
The next lemma states that listing the visits to unstable points is a low entropy event. 

\begin{lem}\label{lem: visit times to unstable points have small entropy} 
	Let $\varepsilon>0$. Then there exists $T\ge 1$ such that the following holds. For any finite subsets $R\subseteq B$ and $L\subseteq A$ with $e_B\in R$ and $e_A\in L$, and $t_0\ge T$, it holds that $H(\mathcal{V}_s)<\varepsilon n$ for every $n\ge 1$, where we denote $s=\lfloor n/t_0\rfloor t_0$.
\end{lem}
\begin{proof}
	Let $\varepsilon>0$. For any $R$ and $L$ as above and an arbitrary value of $t_0$, the visit times to unstable points $\mathcal{V}_s$ form a subset of $\{1,2,\ldots, \lfloor n/t_0\rfloor\}$. Then $H(\mathcal{V}_s)\le \frac{n}{t_0}\log(2)$, so that it suffices to choose $T>\frac{\log(2)}{\varepsilon}$ at the beginning and $t_0\ge T$.
\end{proof}

\subsection{Lamp increments at unstable elements}\label{subsection: lamp increments at unstable elements}
So far, we have proved that, conditioned on the $t_0$-coarse trajectory $\mathcal{P}_n^{t_0}$, it is possible to reveal the positions in the base group for which the lamp configuration has not yet stabilized, together with the time instants in which the lamp state at these positions could have been modified, while adding a small amount of entropy to our process. We will now see that the same holds for the values of the lamp increments that modified these positions in these time instants.

\begin{defn}\label{def: increments at unstable points}
	Let $t_0\ge 1$, and consider finite subsets $R\subseteq B, L\subseteq A$ with $e_B\in R$ and $e_A\in L$. Let $j\in \{1,2,\ldots, \lfloor n/t_0 \rfloor\}$ and $i\in I_j$. Let $n\ge 1$ and denote $s=\lfloor n/t_0\rfloor t_0$. For $b\in \mathcal{U}_s$, let us define the \emph{unstable increment in $b$ at instant $i$} by
	\begin{equation*}
		\Delta_s(b,i)\coloneqq \begin{cases}
			\varphi_{i}(b)^{-1} \varphi_{i+1}(b),&\text{if }I_j\text{ is an }(R,L)\text{-good interval, and}\\
			\star, \text{ otherwise.}
		\end{cases}
	\end{equation*}
\end{defn}


The next lemma corresponds to Step (\labelcref{step: step 9}) of the sketch of the proof (Subsection \ref{subsection: sketch of the proof}).
\begin{lem}\label{lem: increments at unstable points have small entropy}	Let $t_0\ge 1$, and consider finite subsets $R\subseteq B$ and $L\subseteq A$ with $e_B\in R$ and $e_A\in L$. For every $\varepsilon>0$ there exists $K_4\ge 0$ such that for every $n\ge 1$,
	\begin{equation*}
		H(\Delta_s \mid \mathcal{P}_n^{t_0}\vee \mathcal{U}_s\vee \mathcal{V}_s)<\varepsilon n +K_4, \text{ where we denote }s=\lfloor n/t_0\rfloor t_0.
	\end{equation*}
	
\end{lem}
\begin{proof}
	Let us consider $t_0$, $R$ and $L$ as above. Consider $\varepsilon>0$. Note that for every $1\le j \le \lfloor n/t_0\rfloor$, $i\in I_j$ and $b\in \mathcal{U}_s$, the variable $\Delta_s(b,i)$ can take a total of at most $|L|^{|R^{t_0}|}+1$ different values. Using Lemma \ref{lem: not many visits in expectation} we can find $K_3\ge 0$ such that 
	\begin{equation*}
		\E(|\mathcal{V}_s|\mid \mathcal{P}_n^{t_0})<\frac{\varepsilon n}{\log\left(|L|^{|R^{t_0}|}+1\right)}+K_3.
	\end{equation*}
	Now we have
	\begin{align*}
		H(\Delta_s \mid \mathcal{P}_n^{t_0}\vee \mathcal{U}_s\vee \mathcal{V}_s)&\le \E(|\mathcal{V}_s|\mid \mathcal{P}_n^{t_0})\log\left(|L|^{|R^{t_0}|}+1\right)\le \varepsilon n +K_3\log\left(|L|^{|R^{t_0}|}+1\right).
	\end{align*}
	The statement of the lemma follows by choosing $K_4\coloneqq K_3\log\left(|L|^{|R^{t_0}|}+1\right)$.
\end{proof}

\subsection{Proof of Proposition \ref{prop: the lamp config INSIDE the neighborhood has small entropy}}\label{subsection: proof of LAMPS INSIDE}

Now we prove the main result of this section.

\begin{proof}[Proof of Proposition \ref{prop: the lamp config INSIDE the neighborhood has small entropy}]
	Let $\varepsilon>0$. Using Lemma \ref{lem: visit times to unstable points have small entropy} we can find $T\ge 1$ such that for any $R$ and $L$ as above, if $t_0>T$, then $H(\mathcal{V}_s)<\frac{\varepsilon}{3}n$ for every $n\ge 1$. Furthermore, let us consider a constant $K_4\ge 0$ such that
	\begin{equation*}
		H(\Delta_s \mid \mathcal{P}_n^{t_0}\vee \mathcal{U}_s\vee \mathcal{V}_s)<\frac{\varepsilon}{3} n +K_4
	\end{equation*}
	
	for every $n\ge 1$,	which exists thanks to Lemma \ref{lem: increments at unstable points have small entropy}.
	Next, let us choose $N_1\ge 1$ such that if $n\ge N_1$, then $H(\mathcal{U}_s\mid \mathcal{P}_n^{t_0})<\frac{\varepsilon}{3}n$.
	
	In order to simplify our notation, let us denote $\mathcal{N}_n=\mathcal{N}_n(t_0,R)$ and $\beta_n=\beta_n(t_0,R,L)$. Recall the definition of $\Phi_{n}^{\mathrm{in}}(t_0,R)$ as the the lamp configuration inside the coarse neighborhood (Definition \ref{def: entropy in and out lamps}). Note that the value of $\varphi_s$ at an element $b\in \mathcal{N}_n$ is completely determined by the lamp increments that affect position $b$, and this information is completely contained in $\beta_n$ together with $\Delta_s$. This implies that $		H_{A^{B}}\left(\Phi_{n}^{\mathrm{in}}(t_0,R)\Big|\mathcal{P}_n^{t_0}\vee \beta_n\vee \Delta_s\right)=0.$
	
	With the above, we can conclude that for every $n\ge N_1$, it holds that
	\begin{align*}
		H_{A^{B}}\left(\Phi_{n}^{\mathrm{in}}(t_0,R)\Big| \mathcal{P}_n^{t_0}\vee \beta_n\right)&\le H_{A^{B}}\left(\Phi_{n}^{\mathrm{in}}(t_0,R)\Big|\mathcal{P}_n^{t_0}\vee \beta_n\vee \Delta_s\right)+H\left(\Delta_s\Big| \mathcal{P}_n^{t_0}\vee \mathcal{U}_s\vee \mathcal{V}_s\right)+\\&\hspace{30pt}+H\left(\mathcal{V}_s\right)+H\left(\mathcal{U}_s\Big| \mathcal{P}_n^{t_0}\right) \\
		&\le 0 + \frac{\varepsilon}{3} n +K_4 +\frac{\varepsilon}{3} n + \frac{\varepsilon}{3} n=\varepsilon n +K_4.
	\end{align*}
\end{proof}
\section{Proofs of the main theorems}\label{section: proofs of the main theorems}
In this section we present the proofs of Theorem \ref{thm: lamplighter Poisson boundary} and Theorem \ref{thm: Poisson for finite first moment}. Both proofs have a similar structure, and the main difference between them is that for Theorem \ref{thm: Poisson for finite first moment} we need to show that the additional hypothesis of a finite first moment implies that the coarse trajectory on the base group has low entropy, when conditioned on the limit lamp configuration. This fact relies on Proposition \ref{prop: identification of B with finite first moment}, which is proved in the next subsection.

The following proposition will be used in both proofs, and it states that the entropy of the position of the $\mu$-random walk at time $n$ is small, when conditioned on the $t_0$-coarse trajectory and the increments that occurred during the $(R,L)$-bad intervals. This is the combination of the main results of Sections \ref{section: the coarse trajectory} and \ref{section: entropy inside the coarse neighborhood}, where we estimated the entropy of the lamp configuration at instant $n$ outside of the $(t_0,R)$-coarse neighborhood and inside of the $(t_0,R)$-coarse neighborhood, respectively.

\begin{prop} \label{prop: partitions pin down the n-th instant} For any $\varepsilon>0$, there exists $T\ge 1$ such that the following holds. Consider finite subsets $R\subseteq B$ and $L\subseteq A$ with $e_B\in R$ and $e_A\in L$, and $t_0\ge T$. Then there exist $N\ge 1$ and $K\ge 0$ such that for every  $n\ge N$,
	\begin{equation*}\label{eq: partitions pin down the nth instant}
		H_{A^B}\left(\varphi_n,X_n\mid \mathcal{P}_n^{t_0}\vee\beta_n(t_0,R,L)\right)<\varepsilon n +K.
	\end{equation*}
\end{prop}
\begin{proof}
	
	Indeed, consider $\varepsilon>0$ and $T\ge 1$ given by Proposition \ref{prop: the lamp config INSIDE the neighborhood has small entropy}. Then for every choice of $R$, $L$ and $t_0$ as above, there are $N\ge 1$ and $K\ge 0$ such that Equation \eqref{eq: lamp config INSIDE} holds for every $n\ge N$. Let us write $\beta_n=\beta_n(t_0,R,L)$ and $\mathcal{N}_n=\mathcal{N}_n(t_0,R)$ in order to have less notation.
	
	Using Lemma \ref{lem: it suffices to guess the last multiple of t0}, Lemma \ref{lem: lamp config outside coarse neighborhood has small entropy} and the fact that $H(X_s\mid \mathcal{P}_{n}^{t_0})=0$, we have that
	\begin{align*}
		H_{A^B}(\varphi_n,X_n\mid \mathcal{P}_n^{t_0}\vee\beta_n)&\le H_{A^B}(\varphi_s\mid \mathcal{P}_n^{t_0}\vee\beta_n)+H_{A^B}(\varphi_n,X_n\mid \mathcal{P}_n^{t_0}\vee\beta_n\vee\varphi_s)\\
		&= H_{A^B}(\varphi_s\mid \mathcal{P}_n^{t_0}\vee\beta_n) +0\\
		&\le H_{A^B}(\varphi_s|_{B\backslash \mathcal{N}_n}\mid \mathcal{P}_n^{t_0}\vee\beta_n)+H_{A^B}(\varphi_s|_{\mathcal{N}_n}\mid \mathcal{P}_n^{t_0}\vee\beta_n)\\
		&\le 0+\varepsilon n +K.
	\end{align*}
\end{proof}

Now we will prove Theorem \ref{thm: lamplighter Poisson boundary}. The proof of Theorem \ref{thm: Poisson for finite first moment} is given below in Subsection \ref{subsection: measures with a finite first moment}.

\begin{proof}[Proof of Theorem \ref{thm: lamplighter Poisson boundary}]
	
	Thanks to Theorem \ref{thm: conditional entropy alternative formulation}, it suffices to prove that 
	\begin{equation*}
		\lim_{n\to \infty}\frac{H_{A^B\times \partial B}(\varphi_n,X_n)}{n}=0.
	\end{equation*}
	
	Let $\varepsilon>0$. Let us use Lemma \ref{lem: coarse traj has small entropy} to find $T\ge 1$ such that for every $t_0>T$, it holds that for every $n\ge 1$, $H_{\partial B}(\mathcal{P}_n^{t_0})<\frac{\varepsilon}{3} n$.
	
	Thanks to Proposition \ref{prop: partitions pin down the n-th instant} we can find $T\ge 1$ large enough such that in addition to the conditions of the previous paragraph, the following holds for every $t_0>T$. For any choice of finite subsets $R\subseteq B$ and $L\subseteq A$ with $e_B\in R$ and $e_A\in L$, there exist $N\ge 1$ and $K\ge 0$ such that for every $n\ge N$ we have $H_{A^B}\left(\varphi_n,X_n\mid \mathcal{P}_n^{t_0}\vee\beta_n(t_0,R,L)\right)<\frac{\varepsilon}{3} n +K.$ Additionally, by virtue of Lemma \ref{lem: bad increments have small entropy} there exist $R$ and $L$ as above and $K_1\ge 0$, such that $H\left(\beta_n(t_0,R,L)\right)<\frac{\varepsilon}{3}n+K_1 \text{ for every } n\ge 1.$
	
	Consider the associated values of $N$ and $K$ given by Proposition \ref{prop: partitions pin down the n-th instant}. Then for every $n\ge N$, we can use Lemma \ref{lem: basic properties entropy} to obtain
	\begin{align*}
		H_{A^B\times \partial B}\left(\varphi_n,X_n	\right)&\le H_{A^B\times \partial B}\left(\varphi_n,X_n \mid \mathcal{P}_n^{t_0}\vee \beta_n(t_0,R,L)\right)+ H_{A^B\times \partial B}( \mathcal{P}_n^{t_0}\vee \beta_n(t_0,R,L))\\
		&\le H_{A^B\times \partial B}\left(\varphi_n,X_n \mid \mathcal{P}_n^{t_0}\vee \beta_n(t_0,R,L)\right)+H_{A^B\times \partial B}\left(\beta_n(t_0,R,L)\right)+\\
		& \phantom{=H_{A^B\times \partial B}\left(\varphi_n,X_n \mid \mathcal{P}_n^{t_0}\vee \beta_n(t_0,R,L)\right)}+H_{A^B\times \partial B}(\mathcal{P}_n^{t_0}).
	\end{align*}
	
	Next, using Lemma \ref{lem: conditional entropy is at most the usual entropy} we have the following upper bounds: 
	\begin{equation*}
		H_{A^B\times \partial B}(\mathcal{P}_n^{t_0})\le H_{\partial B}(\mathcal{P}_n^{t_0}), \text{ and }H_{A^B\times \partial B}\left(\beta_n(t_0,R,L)\right)\le H\left(\beta_n(t_0,R,L)\right).
	\end{equation*}
	Combining this with the above, we see that
	\begin{align*}
		H_{A^B\times \partial B}\left(\varphi_n,X_n	\right)&\le H_{A^B\times \partial B}\left(\varphi_n,X_n \mid \mathcal{P}_n^{t_0}\vee \beta_n(t_0,R,L)\right)+H\left(\beta_n(t_0,R,L)\right)+H_{\partial B}(\mathcal{P}_n^{t_0}).\\
		&\le \frac{\varepsilon}{3}n+K+\frac{\varepsilon}{3}n+K_1+\frac{\varepsilon}{3}n=\varepsilon n+K+K_1.
	\end{align*}
	
	We conclude that for an arbitrary $\varepsilon>0$, we have $\limsup_{n\to \infty} \frac{H_{A^B\times \partial B}\left(\varphi_n,X_n	\right)}{n}\le\varepsilon$, and hence $
		\lim_{n\to \infty} \frac{H_{A^B\times \partial B}\left(\varphi_n,X_n	\right)}{n}=0.$
\end{proof}

\begin{proof}[Proof of Corollary \ref{cor: finite first moment Zd poisson boundary}]
	It is proved in \cite[Theorem 3.3]{Kaimanovich1991} that if $\mu$ is a probability measure on $A\wr \Z^d$ such that its projection to $\Z^d$ defines a transient random walk, then the lamp configurations stabilize along almost every trajectory to a limit lamp configuration (this holds more generally for any probability measure on $A\wr B$ with a finite first moment that induces a transient random walk on the base group $B$ \cite[Lemma 1.1]{Erschler2011}). Additionally, it is a well-known fact that probability measures with a finite first moment have finite entropy \cite{Derriennic1986} (see also \cite[Lemma 2.2.2]{Kaimanovich2001}). The hypotheses of Theorem \ref{thm: lamplighter Poisson boundary} are thus verified. Since the Poisson boundary of the induced random walk on $\Z^d$ is trivial \cite{Blackwell1955} (see also \cite{DoobSnellWilliamson1960} and \cite{ChoquetDeny1960}), we obtain the desired result.
\end{proof}

\subsection{Proof of Theorem \ref{thm: Poisson for finite first moment}}\label{subsection: measures with a finite first moment}
The objective of this subsection is to prove Theorem \ref{thm: Poisson for finite first moment}. The proof of this result is analogous to that of Theorem \ref{thm: lamplighter Poisson boundary}, with the only difference being that Lemma \ref{lem: coarse traj has small entropy} is replaced by Proposition \ref{prop: identification of B with finite first moment} below. Intuitively, it states that, with the additional assumption of a finite first moment, the $t_0$-coarse trajectory has small entropy when conditioned on the boundary of limit lamp configurations.
\begin{prop}\label{prop: identification of B with finite first moment}
Let $\mu$ be a probability measure on $A\wr B$ with $H(\mu)<\infty$ and such that \ $\langle\supp{\mu}\rangle_{+}$ contains two distinct elements with the same projection to $B$. Suppose that $\mu$ induces a transient random walk on $B$. Then for every $\varepsilon >0$, there exists $T\ge 1$ such that for every $t_0\ge T,$
	\begin{equation*}
		H_{A^B}(\mathcal{P}_n^{t_0})<\varepsilon n, \text{ for every } n\ge 1.
	\end{equation*}
\end{prop}

In what remains of this subsection we prove Proposition \ref{prop: identification of B with finite first moment}.

Let us first recall the relation between the asymptotic entropy, the growth of the group and the speed of a random walk driven by a step distribution with a finite first moment.

Let $G$ be a finitely generated group and let $|\cdot |$ be the word length with respect to some finite generating set of $G$. Let $\mu$ be a probability measure on $G$ with a finite first moment, and denote by $\{w_n\}_{n\ge 0}$ the $\mu$-random walk on $G$. It is a consequence of Kingman's subadditive ergodic theorem that the limit
\begin{equation} \label{eq: speed rw}
	\ell\coloneqq\lim_{n\to \infty}\frac{\E(|w_n|)}{n}
\end{equation}
exists, and that for $\P$-almost every trajectory of the $\mu$-random walk one has $\ell=\lim_{n\to \infty}\frac{|w_n|}{n}$. The value of $\ell$ is called the \emph{speed} of the random walk $(G,\mu)$. Additionally, let us denote $v\coloneqq\lim_{n\to \infty} \frac{1}{n}\log\Big(\left| \left\{ g\in G\mid |g|\le n \right\}\right|\Big)$ the \emph{exponential growth rate} of $G$. Recall also Definition \ref{defn: asymptotic entropy} for the asymptotic entropy of $(G,\mu)$.  The so-called ``fundamental inequality'' of Guivarc'h states that $h_{\mu}\le \ell v$ \cite[Proposition 2]{Guivarch1980}. 

In order to prove Proposition \ref{prop: identification of B with finite first moment}, we would like to show that the Poisson boundary of the $\mu_B$-random walk on $B$ is completely described by the stabilization of the limit lamp configuration. However, we cannot directly apply the conditional entropy criterion. Indeed, this space is not a $\mu_B$-boundary, so applying the usual results is not correct. Nonetheless, we can modify the definitions and follow the same ideas of Kaimanovich's conditional entropy criterion to prove our result. The following is similar to the exposition in \cite[Section 4]{Kaimanovich2000}. 

Recall our notation from Subsection \ref{subsection: construction of the boundary} for the $\mu$-random walk $\{(\varphi_n,X_n)\}_{n\ge 0}$, so that $(\varphi_n,X_n)=(f_1,Y_1)\cdots (f_n,Y_n)$, $n\ge 0$, with $\{(f_i,Y_i)\}_{i\ge 1}$ a sequence of independent identically distributed random variables with law $\mu$. Recall also the definitions associated with entropy, partitions and conditional entropy of Subsection \ref{subsection: entropy}. 

We define
$ h\left(B,\mu_B\mid A^B\right)\coloneqq \lim_{n\to \infty}\frac{H_{A^B}(X_n)}{n}, $
which is well-defined since the sequence $\{H_{A^B}(X_n)\}_{n\ge 0}$ is subadditive. For every $\varphi_{\infty}\in A^B$ denote by $p_n^{\varphi_{\infty}}$, $n\ge 0$, the one-dimensional distributions of the measure $\P^{\varphi_{\infty}}.$ It is a consequence of Kingman's subadditive ergodic theorem that for $\P^{\varphi_{\infty}}$-almost every trajectory $\{(\varphi_n,X_n)\}_{n\ge 0}$, the limit
\( h\left(B,\mu_B\mid \varphi_{\infty}\right)\coloneqq -\lim_{n\to \infty}\frac{\log(p_n^{\varphi_{\infty}}(X_n))}{n} \)
is well defined and a constant. Furthermore, we have
\[  h\left(B,\mu_B\mid \varphi_{\infty}\right)=\lim_{n\to \infty}\frac{\E\Big(-\log(p_n^{\varphi_{\infty}}(X_n))\Big)}{n}=\lim_{n\to \infty}\frac{H_{\varphi_{\infty}}(X_n)}{n}, \]

so that after integrating we have 

\begin{equation}\label{eq: relation between entropy AB and entropy varphiinfty}  h\left(B,\mu_B\mid A^B\right)=\int_{A^B} h\left(B,\mu_B\mid \varphi_{\infty}\right)\ d\nu_{\mathcal{L}}(\varphi_{\infty}).
\end{equation}

With this, Proposition \ref{prop: identification of B with finite first moment} will follow from proving that $h\left(B,\mu_B\mid A^B\right)=0$ in Lemma \ref{lem: hAB 0} below.

If the induced random walk on $B$ is Liouville, then it follows from the usual entropy criterion that $h\left(B,\mu_B\right)=0$, which implies the above. Hence, we will suppose from now on that the induced random walk on $B$ is not Liouville. Denote by $\ell$ the speed of the $\mu_B$-random walk on $B$. It follows from the entropy criterion together with the fundamental inequality of Guivarc'h that $\ell>0$. 

%
%

\begin{lem}\label{lem: sets for finite first moment entropy}
Let $\mu$ be a probability measure on $A\wr B$ with a finite first moment and such that $\langle\supp{\mu}\rangle_{+}$ contains two distinct elements with the same projection to $B$. Suppose that the random walk induced by $\mu$ on $B$ is not Liouville.	Then, for $\nu_{\mathcal{L}}$-almost-every $\varphi_{\infty}\in A^B$ there is a family of finite subsets $Q_n\subseteq B$, $n\ge 1$, such that
	\begin{enumerate}
		\item $\limsup_{n\to \infty}\P^{\varphi_{\infty}}(X_n\in Q_n)>0$, and
		\item $\lim_{n\to \infty}\frac{1}{n}\log|Q_n|=0.$
	\end{enumerate}
\end{lem}
\begin{proof}
	Let $S$ be an arbitrary finite generating set of $B$ and let us denote by $d_S(\cdot,\cdot)$ and $|\cdot|_S$ the associated word metric and word length. Given an element $b\in B$ and a non-empty subset $F\subseteq B$, we denote $d_S(b,F)\coloneqq \min\{d_S(b,f)\mid f\in F\}$.
	
	We first note that the speed $\ell$ of the $\mu_B$-random walk, as defined in Equation \eqref{eq: speed rw}, satisfies that $\P$-almost surely we have that for every $n$ large enough,
	\begin{equation}\label{item: position}
		(\ell-\varepsilon) n\le |X_n|_S\le (\ell +\varepsilon)n.
	\end{equation} 
	Additionally, the finite first moment hypothesis implies that $\P$-almost surely for all $n$ large enough,
	\begin{equation}\label{item: lamp increments}
		\text{the lamp increment }f_{n} \text{ only modifies lamps at a }d_S\text{-distance at most }\varepsilon n.
	\end{equation}
	
	Since $\mu$ has a finite first moment and is adapted, it is a consequence of Kingman's subadditive ergodic theorem that there exists $C\ge 0$ such that
\( 
\lim_{n\to \infty}\frac{|\supp{\varphi_n}|}{n}=C.
\)
This is proved in \cite[Lemma 2.1]{Erschler2011} (see also \cite[Theorem 3.1]{Gilch2008}).
Then, it holds that that $\P$-almost surely for every $n$ large enough, we have
	\begin{equation} \label{item: support lamps}
		(C-\varepsilon)n\le |\supp{\varphi_n}|\le (C+\varepsilon)n
	\end{equation}

	Let us choose $N\ge 1$ large enough so that Equations \eqref{item: position}, \eqref{item: lamp increments} and \eqref{item: support lamps} hold for all $n\ge N$. 
	
	We will now argue that with positive probability, the position of the random walk $X_n$ in the base group is close to a position $b\in B$ on which the limit lamp configuration is turned on: $\varphi_{\infty}(b)\neq e_A$. 
	
	Since we are assuming that $\langle\supp{\mu}\rangle_{+}$ contains two distinct elements with the same projection to $B$, there exist $k\ge 1$ and $g,g^{\prime } \in \supp{\mu^{*k}}$ with the same projection to $B$, and with lamp configurations $f,f^\prime\in \bigoplus_{B}A$ such that $f\neq f^{\prime}$. Let $D>0$ be large enough such that $\supp{f}$ and $\supp{f^{\prime}}$ are contained in the ball of radius $D$ centered at $e_B$. 
	
		
		Now let us look at sample paths $w=(w_1,w_2,\ldots) \in (A\wr B)^{\infty}$, and condition on all of the increments except for those that occur at instants $n+1,n+2,\ldots,n+k$, and such that the product of the increments $g_{n+1}g_{n+2}\cdots g_{n+k}$ is either $g$ or $g^{\prime}$. For such sample paths, the values of the limit lamp configuration $\varphi_{\infty}$ are always the same except possibly at positions within the ball of radius $D$ of $X_n$. Since the lamp configurations of $g$ and $g^{\prime}$ are distinct, then in at least one of the two limit lamp configurations there will be $b\in B$ with $d_S(X_n,b)\le D$ and with $\varphi_{\infty}(b)\neq e_A$. This implies that for all $n\ge 1$ we have $	\P^{\varphi_{\infty}}\Big(d_S(X_n,\supp{\varphi_{\infty}})\le D\Big)\ge \min\{ \mu^{*k}(g),\mu^{*k}(g^{\prime}) \}>0 $, and hence
	
	\begin{equation}\label{eq: Xn is close to a lit lamp} \limsup_{n\to \infty}\P^{\varphi_{\infty}}\Big(d_S(X_n,\supp{\varphi_{\infty}})<D\Big) >0.
		\end{equation}

%
%
%
	
	Let us define for every $n\ge 1$ the finite subset $Q_n\subseteq B$ as the set of elements $b\in B$ such that $(\ell-\varepsilon)n\le|b|_S\le (\ell+\varepsilon)n$, and $d_S(b,\supp{\varphi_\infty})<D$. Item \eqref{item: position} above together with Equation \eqref{eq: Xn is close to a lit lamp} imply that $\limsup_{n\to \infty}\P^{\varphi_{\infty}}(X_n\in Q_n)>0$. It remains to estimate the size of $Q_n$ for $n$ large enough. 
	
	Let us denote the ball of radius $r$ in $B$ by $\mathrm{Ball}(r)\coloneqq \{b\in B\mid |b|_S\le r\}$. Note that any element $b\in Q_n$ is at distance at most $D$ of some $x\in \supp{\varphi_{\infty}}$, such that $(\ell-\varepsilon)n-D\le |x|_S\le (\ell+\varepsilon)n+D$. Thus, we have that 
	\begin{equation}\label{eq: sets balls support} |Q_n|\le \left|\supp{\varphi_{\infty}}\cap \Big(\mathrm{Ball}((\ell+\varepsilon)n +D)\backslash\mathrm{Ball}(\ell-\varepsilon)n -D) \Big)\right|. \end{equation}
	
	Recall that for large enough $n$, Equations \eqref{item: position}, \eqref{item: lamp increments} and \eqref{item: support lamps} above hold. Thanks to this, the last instant $k_1$ that some lamp inside $\mathrm{Ball}((\ell+\varepsilon)n +D)$ is modified satisfies $(\ell-\varepsilon)k_1-\varepsilon n\le (\ell+\varepsilon)n+D$, so that $ k_1\le \frac{\ell+2\varepsilon}{\ell-\varepsilon}n+\frac{D}{\ell-\varepsilon}. $ Similarly, any modification to a lamp outside of $\mathrm{Ball}(\ell-\varepsilon)n -D)$ can only occur after the time instant $k_2$ that satisfies $k_2(\ell+\varepsilon)+\varepsilon n\ge (\ell-\varepsilon)n-D$, so that $ k_2\ge \frac{\ell-2\varepsilon}{\ell+\varepsilon}n-\frac{D}{\ell+\varepsilon}. $
	
	We conclude from the above together with Item \eqref{item: support lamps} that for $n$ large enough, the set $\supp{\varphi_{\infty}}$ contains at most $(C+\varepsilon)\left(\frac{\ell+2\varepsilon}{\ell-\varepsilon}n+\frac{D}{\ell-\varepsilon} \right)$ elements inside $\mathrm{Ball}((\ell+\varepsilon)n +D)$, and at least $(C-\varepsilon)\left(\frac{\ell-2\varepsilon}{\ell+\varepsilon}n-\frac{D}{\ell+\varepsilon} \right)$ inside $\mathrm{Ball}((\ell-\varepsilon)n -D)$.
	
	Now we can use Equation \eqref{eq: sets balls support} to obtain
	\begin{align*}
		|Q_n|&\le (C+\varepsilon)\left(\frac{\ell+2\varepsilon}{\ell-\varepsilon}n+\frac{D}{\ell-\varepsilon} \right) -(C-\varepsilon)\left(\frac{\ell-2\varepsilon}{\ell+\varepsilon}n-\frac{D}{\ell+\varepsilon} \right)=C_1n+C_2,
	\end{align*}
	for some appropriate constants $C_1,C_2>0$ whose explicit value is not relevant. From this we obtain that $\lim_{n\to \infty}\frac{1}{n}\log|Q_n|=0$, hence finishing the proof.
\end{proof}
\begin{lem}\label{lem: hAB 0}
	We have $h\left(B,\mu_B\mid A^B\right)=0$.
\end{lem}
\begin{proof}
	Looking for a contradiction, let us suppose that $h\left(B,\mu_B\mid A^B\right)>0$. Thanks to Equation \eqref{eq: relation between entropy AB and entropy varphiinfty}, it must hold that for $\varphi_{\infty}$ in a measurable subset $T\subseteq A^B$ of positive $\nu_{\mathcal{L}}$-measure we have $h\coloneqq h\left(B,\mu_B\mid \varphi_{\infty}\right)>0$.
	
	For $\varphi_{\infty}\in T$ and  $n\ge 1$, let us define the set $
		E_n\coloneqq \left \{ b\in B\mid p^{\varphi_{\infty}}_n(b)< e^{-nh/2} \right \}.$ Since for $\P^{\varphi_{\infty}}$-almost-every trajectory $\{(\varphi_n,X_n)\}_{n\ge 1}$ we have $\displaystyle\lim_{n\to \infty}\frac{1}{n}\log p^{\varphi_{\infty}}_n(X_n)=-h$,  then it holds that $\lim_{n\to \infty}\P^{^{\varphi_{\infty}}}(X_n\in E_n)=1$ for every $\varphi_{\infty} \in T$.
	
	By virtue of Lemma \ref{lem: sets for finite first moment entropy} we can find finite subsets $Q_n\subseteq B$, $n\ge 1$, such that \begin{equation*}\lim_{n\to \infty}\frac{1}{n}\log|Q_n|= 0\text{ and }\limsup_{n\to \infty} \P^{\varphi_{\infty}}(X_n\in Q_n)>0.
	\end{equation*}
	Then, we see that $\P^{\varphi_{\infty}}\left(X_n\in E_n\cap Q_n\right)=\sum_{b\in E_n \cap Q_n}\P^{\varphi_{\infty}}\left(X_n=b\right)\le |Q_n|e^{-nh/2}.$ 	However, we have that $\lim_{n\to \infty} |Q_n|e^{-nh/2}=0$. Indeed, 
	\begin{equation*}
		\log(|Q_n|e^{-nh/2})=n\left( \frac{1}{n}\log|Q_n| -\frac{h}{2} \right)\xrightarrow[n\to \infty]{}-\infty.
	\end{equation*}
	This implies that $\lim_{n\to \infty}\P^{\varphi_{\infty}}\left(X_n\in E_n\cap Q_n\right)=0$. Furthermore, since we have that $\lim_{n\to \infty}\P^{^{\varphi_{\infty}}}(X_n\in E_n)=1$, it must hold that $\lim_{n\to \infty}\P^{\varphi_{\infty}}\left(X_n\in Q_n\right)=0$. This contradicts our choice of the sets $Q_n$.
\end{proof}


\begin{proof}[Proof of Proposition \ref{prop: identification of B with finite first moment}]
	
	Thanks to Lemma \ref{lem: hAB 0}, we have that 
	\( \lim_{n\to \infty}\frac{H_{A^B}(X_n)}{n}=0. \) This implies that for every $\varepsilon>0$, there exists $T\ge 1$ such that $	H_{A^B}(X_{t_0})<\varepsilon t_0$ for every $t_0\ge T$. Afterward, the proof follows analogously to the proof of Lemma \ref{lem: coarse traj has small entropy}. Indeed, let $n\ge 1$, $t_0\ge T$, and denote $s=\lfloor n/t_0\rfloor.$ Then
	\begin{equation*}
		H_{A^B}(X_{t_0},X_{2t_0},\ldots, X_{st_0})\le\sum_{j=0}^{s-1}H_{A^B}\left(X_{(j+1)t_0}\mid X_{jt_0}\right)
		= sH_{A^B}(X_{t_0})
		< s\varepsilon t_0\le \varepsilon n.
	\end{equation*}		
\end{proof}
We remark that the assumption that $\langle \supp{\mu}\rangle_{+}$ contains two distinct elements with the same projection to $B$ in Theorem \ref{thm: Poisson for finite first moment} and Proposition \ref{prop: identification of B with finite first moment} is necessary, as the following example shows.
\begin{exmp}\label{example: conjugate base group}
	Consider the wreath product $\Z/2\Z\wr F_2$, where $F_2$ is the free group of rank $2$. We denote $\Z/2\Z=\{[0]_2,[1]_2\}$, where $[0]_2\in \Z/2\Z$ is the identity element. Denote by $\delta_0:F_2\to \Z/2\Z$ the lamp configuration defined by $\delta_0(e_{F_2})=[1]_2$ and $\delta_0(x)=[0]_2$ for $x\in F_2\backslash \{e_{F_2}\}$. Let us also denote by $\{a,b\}$ a free generating set of $F_2$. Let $\mu$ be the probability measure on $\Z/2\Z\wr F_2$ defined by $\mu(\delta_0 x \delta_0)=1/4$ for each $x \in \{ a,a^{-1},b,b^{-1}\}$. Then the projection of the $\mu$-random walk on $\Z/2\Z\wr F_2$ to $F_2$ is a simple random walk, and in particular this implies that the Poisson boundary is non-trivial. However, note that the limit lamp configuration $\varphi_{\infty}$ along any sample path satisfies that $\varphi_{\infty}(x)=[0]_2$ for all $x\in F_2\backslash\{ e_{F_2}\}$. The $\mu$-boundary of infinite lamp configurations is then trivial, and hence it is not the Poisson boundary.
\end{exmp}

\section{Groups of the form $F_d/[N,N]$ and free solvable groups} \label{section: free solvable groups}
\subsection{The Magnus embedding}\label{subsection: magnus}
Let $F_d$ be a free group of rank $d\ge 1$, and consider a normal subgroup $N\lhd F_d$. We now recall the Magnus embedding of $F_d/[N,N]$ into the wreath product $\Z^d\wr F_d/N$. We follow the expositions of \cite[Subsections 2.2 -- 2.5]{MyasnikovRomankovUshakovVershik2010} and \cite[Section 2]{SaloffCosteZheng2015}, which are our main references for this subsection.

We first introduce the notation that we will use. We denote by $\pi:F_d\to F_d/N$, $\widetilde{\pi}:F_d/[N,N]\to F_d/N$ and $\theta:F_d\to F_d/[N,N]$ the respective quotient maps. Note that $\pi = \widetilde{\pi} \circ \theta$. We denote by $\Z(F_d/N)$ and $\Z(F_d)$ the group rings with integer coefficients of $F_d/N$ and $F_d$, respectively. Let us also denote by $\pi:\Z(F_d)\to \Z(F_d/N)$ the linear extension of the quotient map $\pi:F_d\to F_d/N$.

Let $T$ be a free left $\Z(F_d/N)$-module of rank $d$ with basis $\{t_1,\ldots,t_d\}$, and consider the group of matrices
\[
M(F_d/N)\coloneqq \left\{ \begin{pmatrix}
	g & t \\ 0 & 1
\end{pmatrix}  \mid g\in F_d/N \text{ and }t\in T \right\},
\]
where the group operation is matrix multiplication. One can verify with a direct computation that $T$ is isomorphic to the direct sum $\bigoplus_{F_d/N}\Z^d$, and that $M(F_d/N)$ is isomorphic to $\Z^d\wr F_d/N$. The following theorem describes the Magnus embedding of $F_d/[N,N]$ into $\Z^d\wr F_d/N$.
\begin{thm}[{\cite{Magnus1939}}]\label{thm: original formulation Magnus embedding} Let $F_d$ be a free group of rank $d\ge 1$, and consider $X=\{x_1,\ldots,x_d\}$ a free basis of $F_d$. Let $N\lhd F_d$ be a normal subgroup and denote by $\pi:F_d\to F_d/N$ the quotient map. Then the homomorphism $\phi:F_d\to M(F_d/N)$ defined by
	\begin{equation*}
		\phi(x_i)\coloneqq \begin{pmatrix}
			\pi(x_i) & t_i \\
			0 & 1
		\end{pmatrix}, \ \ \ \text{ for }i=1,\ldots,d,
	\end{equation*}
satisfies $\mathrm{ker}(\phi)=[N,N]$. Hence, $\phi$ induces a monomorphism $\widetilde{\phi}:F_d/[N,N]\to M(F_d/N)$.
\end{thm}

In other words, the Magnus embedding $\widetilde{\phi}:F_d/[N,N]\to M(F_d/N)\cong \Z^d\wr F_d/N$ can be written as $\widetilde{\phi}(g)=(\varphi(g),\widetilde{\pi}(g))$ for each $g\in F_d/[N,N]$, where $\varphi(g)\in \bigoplus_{F_d/N}\Z^d$ and $\widetilde{\pi}:F_d/[N,N]\to F_d/N$ is the quotient map. In order to describe more precisely the map $\varphi:F_d/[N,N]\to \bigoplus_{F_d/N}\Z^d$, we need to introduce Fox derivatives.

Consider the group ring $\Z(F_d)$, and let us denote by $\varepsilon:\Z(F_d)\to \Z$ the extension of the trivial group homomorphism $F_d\to 1$. A map $D:\Z(F_d)\to \Z(F_d)$ is called a \emph{left derivation} if it satisfies:
\begin{itemize}
	\item $D(u+v)=D(u)+D(v)$, for every $u,v\in \Z(F_d)$, and
	\item $D(uv)=D(u)\varepsilon(v)+uD(v)$, for every $u,v\in \Z(F_d)$.
\end{itemize}

The following two theorems are due to Fox \cite{FoxI,FoxII,FoxIII,FoxIV,FoxV}.
\begin{thm}[Fox] Let $F_d$ be a free group of rank $d\ge 1$, and consider $X=\{x_1,\ldots,x_d\}$ a free basis of $F_d$. Then for each $i=1,\ldots,d$ there is a unique left derivation $\partial_i:\Z(F_d)\to \Z(F_d)$ such that 
	\[\partial_i(x_j)=\begin{cases}1, &\text{ if }i=j, \text{ and}\\ 0, &\text{ otherwise.}
	\end{cases}\]

 Furthermore, let $N\lhd F_d$ a normal subgroup and denote by $\pi:F_d\to F_d/N$ the quotient map. Then for every $u\in F_d$ it holds that  $u\in [N,N] \Longleftrightarrow \pi(\partial_i u) = 0 \text{ for every }i=1,\ldots,d.$

\end{thm}

\begin{thm}[Fox]\label{thm: FOX formulation Magnus embedding}  Let $F_d$ be a free group of rank $d\ge 1$ and let $N\lhd F_d$ a normal subgroup. Denote by $\pi:F_d\to F_d/N$ and by $\theta: F_d\to F_d/[N,N]$ the associated quotient maps. Then the Magnus embedding $\widetilde{\phi}:F_d/[N,N]\to M(F_d/N)$ can be expressed as 
	\begin{equation*}
		\widetilde{\phi}(g)=\begin{pmatrix}
			\widetilde{\pi}(g) & \sum_{i=1}^d \pi(\partial_i u)\cdot t_i\\ 0 & 1
		\end{pmatrix}, \ \ \ \text{ for }g\in F_d/[N,N],
	\end{equation*}
	
	where $u\in F_d$ is any element such that $\theta(u)=g$.
\end{thm}

\subsection{Flows}\label{subsection: flows}
We now explain how the elements of the group $F_d/[N,N]$ can be expressed in terms of flows in the Cayley graph of $F_d/N$. We first need to introduce the notation and definitions associated with Cayley graphs and flows.

Let $F_d$ be the free group of rank $d$ and let $N\lhd F_d$ be a normal subgroup. By slightly abusing notation, we will denote by $X=\{x_1,\ldots, x_d\}$ a free generating set of $F_d$, as well as its image on $F_d/N$ via the canonical quotient map $\pi:F_d\to F_d/N$.

The \emph{directed and labeled Cayley graph} $\mathrm{Cay}(F_d/N)$ is the graph with vertex set $V=F_d/N$, and with edges 
\(
\mathrm{Edges}(F_d/N)\coloneqq \{ (g_1,g_2,x)\in V\times V\times \{x_1,\ldots,x_d\}\mid g_2=g_1x \}.
\)
Hence, each edge is of the form $\mathfrak{e}\coloneqq (g,gx_i,x_i)$ for some $g\in F_d/N$ and $i\in \{1,\ldots,d\}$. We call $g=o(\mathfrak{e})$ the \emph{origin} of $\mathfrak{e}$, $gx_i=t(\mathfrak{e})$ the terminus of $\mathfrak{e}$, and $x_i$ the label of $\mathfrak{e}$.
\begin{defn}
	
	\begin{enumerate}
		\item Let $\mathfrak{f}:\mathrm{Edges}(F_d/N) \to \Z$ be a function. For each $g\in F_d/N$ we define the \emph{net flow} $\mathfrak{f}^{*}(g)$ of $\mathfrak{f}$ at $g$ by
		\[
		\mathfrak{f}^{*}(g)\coloneqq\sum_{o(\mathfrak{e})=g}f(\mathfrak{e})-\sum_{t(\mathfrak{e})=g}f(\mathfrak{e}).
		\]
		\item Let $s,t\in F_d/N$. A \emph{flow} with source $s$ and sink $t$ is a function $\mathfrak{f}:\mathrm{Edges}(F_d/N) \to \Z$ such that $f^{*}(g)=0$ for all $g\in F_d/N\backslash \{s,t\}$. If this condition is also satisfied for $g=s$ and $g=t$, we call the flow $\mathfrak{f}$ a \emph{circulation}.
		\item  We say that a flow $\mathfrak{f}$ is \emph{geometric} if the set $\{\mathfrak{e}\in \mathrm{Edges}(F_d/N)\mid \mathfrak{f}(\mathfrak{e})\neq 0 \}$ is finite (i.e.\ $\mathfrak{f}$ is finitely supported), and if $\mathfrak{f}$ is either a circulation or it is a flow with source $s$ and sink $t$ such that $\mathfrak{f}^{*}(s)=1$ and $\mathfrak{f}^{*}(t)=-1$.
	\end{enumerate}
\end{defn}

For each edge $\mathfrak{e}=(g,gx,x)\in \mathrm{Edges}(F_d/N)$, let us introduce its formal inverse $\mathfrak{e}^{-1}=(gx,g,x^{-1})$. Denote by $\mathrm{Edges}(F_d/N)^{\pm 1}$ the set of edges together with their formal inverses.

	A \emph{finite path} $p$ on the Cayley graph $\mathrm{Cay}(F_d/N)$ is a sequence $p=(\mathfrak{e}_1,\mathfrak{e}_2,\dots,\mathfrak{e}_k)$ of edges $\mathfrak{e}_j\in \mathrm{Edges}(F_d/N)^{\pm 1}$, $j=1,\ldots,k$, such that $t(\mathfrak{e}_j)=o(\mathfrak{e}_{j+1})$ for each $j=1,2,\ldots,k-1$. The element $o(\mathfrak{e}_1)$ is called the \emph{origin of $p$} and the element $t(\mathfrak{e}_k)$ is called the \emph{terminus of $p$}. Note that each finite path defines a word in $F_d$, given by concatenating the labels of edges along the path, and reciprocally any word in $F_d$ defines a unique path with origin at the identity element $e_{F_d/N}$.

	Let $p$ be a finite path on the Cayley graph $\mathrm{Cay}(F_d/N)$. Then $p$ defines a geometric flow $\mathfrak{f}_p:\mathrm{Edges}(F_d/N)\to \Z$, given by
	\begin{align*}
	\mathfrak{f}_p(\mathfrak{e})&\coloneqq \text{algebraic number of times that }\mathfrak{e} \text{ is crossed (positively or negatively) along }p\\
	&=(\text{number of times that }\mathfrak{e} \text{ occurs in }p) - (\text{number of times that }\mathfrak{e}^{-1} \text{ occurs in }p),
	\end{align*}
	for each $\mathfrak{e}\in \mathrm{Edges}(F_d/N)$.

For every element $g\in F_d/[N,N]$, choose any word $u\in F_d$ which projects to $g$ and define  the flow associated with $g$ by $\mathfrak{f}_g\coloneqq \mathfrak{f}_{u}.$ The following theorem guarantees that $\mathfrak{f}_g$ is well-defined and does not depend on the choice of the word $u$.

\begin{thm}[{\cite[Theorem 2.7]{MyasnikovRomankovUshakovVershik2010}}]\label{thm: fox flows}
  Let $F_d$ be a free group of rank $d\ge 1$ and let $N\lhd F_d$ be a normal subgroup. Denote by $\theta: F_d\to F_d/[N,N]$ the associated quotient map.
  	Then for any $u,v\in F_d$ it holds that $\theta(u)=\theta(v)$ if and only if $\mathfrak{f}_u=\mathfrak{f}_v$.
\end{thm}

Recall that we denote by $\pi$ the canonical epimorphism $F_d\to F_d/N$ as well as its linear extension $\Z(F_d)\to \Z(F_d/N)$. The following result expresses, for each $i=1,\ldots,d$ and $u\in F_d$, the values $\pi(\partial_i u)$ in terms of the flow $\mathfrak{f}_u$ on the Cayley graph of $F_d/N$ associated with $u$.
\begin{lem}[{\cite[Lemma 2.6]{MyasnikovRomankovUshakovVershik2010}}]\label{lem: fox derivatives in terms of flows}  Let $F_d$ be a free group of rank $d\ge 1$ with a free generating set $\{x_1,\ldots,x_d\}$, and let $N\lhd F_d$ a normal subgroup. Denote by $\pi:\Z(F_d)\to \Z(F_d/N)$ the linear extension of the quotient map $F_d\to F_d/N$. Then for every $u\in F_d$ and any $i=1,\ldots,d$ we have
	\(
	\pi(\partial_i u)=\sum_{g\in F_d/N}\mathfrak{f}_u((g,gx_i,x_i))g.
	\)
\end{lem}

\begin{rem}\label{remark: flow completely determines element of FdNN}
	Note that, for any element $g\in F_d/[N,N]$, the flow $\mathfrak{f}_g$ completely determines the projection $\widetilde{\pi}(g)\in F_d/N.$ Indeed, if $\mathfrak{f}_g$ is a circulation, then $\widetilde{\pi}(g)=e_{F_d/N}$ is the identity element, and otherwise $\widetilde{\pi}(g)$ is the sink of the geometric flow $\mathfrak{f}_g$. This observation together with Lemma \ref{lem: fox derivatives in terms of flows} and Theorem \ref{thm: FOX formulation Magnus embedding} imply that the element $g\in F_d/[N,N]$ is completely determined by the flow $\mathfrak{f}_g$.
\end{rem}

\subsection{Proofs of Corollaries \ref{cor: general description for general Magnus embedding} and \ref{cor: description for free solvable group}}

Now we consider random walks on the group $F_d/[N,N]$, for $N\lhd F_d$ a normal subgroup of the free group $F_d$ of rank $d$. In analogy with the stabilization of the lamp configuration on wreath products (Definition \ref{def: stabilization of lamps}), we define the stabilization of the flows associated with elements of $F_d/[N,N]$.

\begin{defn}\label{def: stabilization of flows}
Let $\mu$ be a probability measure on the group $F_d/[N,N]$. We say that the flows $\{\mathfrak{f}_{w_n}\}_{n\ge 0}$ associated with the $\mu$-random walk $\{w_n\}_{n\ge 0}$ on $F_d/[N,N]$ \emph{stabilize almost surely} if for every $\mathfrak{e}\in \mathrm{Edges}(F_d/N)$, there exists $N\ge 1$ such that $\mathfrak{f}_{w_n}(\mathfrak{e})=\mathfrak{f}_{w_N}(\mathfrak{e})$ for every $n\ge N$. 
\end{defn}
The stabilization of flows on the group $F_d/[N,N]$ is directly related to the stabilization of the lamp configuration in the wreath product $\Z^d\wr F_d/N$ via the Magnus embedding, as we state in the following lemma.
\begin{lem}\label{lem: flows stabilize iff lamps stabilize}
Let $\mu$ be a probability measure on the group $F_d/[N,N]$, and denote by $\widetilde{\phi}:F_d/[N,N]\to \Z^d\wr F_d/N$ the Magnus embedding. Then the flows $\{\mathfrak{f}_{w_n}\}_{n\ge 0}$ associated with the $\mu$-random walk $\{w_n\}_{n\ge 0}$ on $F_d/[N,N]$ stabilize almost surely if and only if the lamp configurations of $\left\{\widetilde{\phi}(w_n) \right\}_{n\ge 0}$ stabilize almost surely. In such a case, the edgewise limit $\mathfrak{f}_{\infty}:\mathrm{Edges}(F_d/N)\to \Z$ of the flows completely determines the pointwise limit $\varphi_{\infty}:F_d/N\to \Z^d$ of the lamp configurations and vice-versa.
\end{lem}
\begin{proof}

For each sample path $\{w_n\}_{n\ge 0}\in (F_d/[N,N])^{\infty}$, let us consider its image $\left\{\widetilde{\phi}(w_n) \right\}_{n\ge 0}\in (\Z^d\wr F_d/N)^{\infty}$ via the Magnus embedding. Theorem \ref{thm: FOX formulation Magnus embedding} states that we can write
	\begin{equation*}
	\widetilde{\phi}(w_n)=\begin{pmatrix}
		\widetilde{\pi}(w_n) & \sum_{i=1}^d \pi(\partial_i u_n)\cdot t_i\\ 0 & 1
	\end{pmatrix},
\end{equation*}

where $u_n\in F_d$ is any element such that $\theta(u_n)=w_n$. In addition, by virtue of Lemma \ref{lem: fox derivatives in terms of flows}, we can write
$
\pi(\partial_i u_n)=\sum_{g\in F_d/N}\mathfrak{f}_{w_n}((g,gx_i,x_i))g, \text{ for each }i=1,\ldots,d.
$ Let us denote by $\varphi_n:F_d/N\to \Z^d$ the lamp configuration of $\widetilde{\phi}(w_n)$. Then for each $g\in F_d/N$, the above implies that
\begin{equation}\label{eq: direct relation lamp flows}
\varphi_n(g)=\sum_{i=1}^d \mathfrak{f}_{w_n}((g,gx_i,x_i))t_i, \text{ for all }n\ge 1.
\end{equation}
Equation \eqref{eq: direct relation lamp flows} directly relies the value of the lamp configuration at an element $g\in F_d/[N,N]$ with the values of the flow at edges $\mathfrak{e}\in \mathrm{Edges}(F_d/N)$ with origin $o(\mathfrak{e})=g$. Since Equation \eqref{eq: direct relation lamp flows} consists of a finite sum, we conclude that the flows stabilize along a sample path if and only if the lamp configuration stabilizes along the image of the sample path via the Magnus embedding. In such a case, we see also from Equation \eqref{eq: direct relation lamp flows}, that the edgewise limit flow $\mathfrak{f}_{\infty}$ and the pointwise limit lamp configuration $\varphi_{\infty}$ satisfy
\begin{equation*}
	\varphi_{\infty}(g)=\sum_{i=1}^d \mathfrak{f}_{\infty}((g,gx_i,x_i))t_i, \text{ for all }g\in F_d/N.
\end{equation*}
We conclude that $\varphi_{\infty}$ completely determines $\mathfrak{f}_{\infty}$ and vice-versa.
\end{proof}

Suppose that the flows stabilize almost surely along sample paths of the $\mu$-random walk on $F_d/[N,N]$. this implies that there is a shift-invariant map $\mathcal{F}:(F_d/[N,N])^{\mathbb{N}}\to \Z^{\mathrm{Edges}(F_d/N)}$ that is measurable and well-defined for $\P$-a.e.\ sample path $\{w_n\}_{n\ge 0}$ of the $\mu$-random walk on $F_d/[N,N]$, given by the edgewise limit of the flows:
\(
\mathcal{F}((w_n)_{n\ge 0})\coloneqq \lim_{n\to \infty} \mathfrak{f}_{w_n}.
\) 
The hitting measure is $\nu_{\mathcal{F}}\coloneqq \mathcal{F}_{*}\P$. Hence, the space $(\Z^{\mathrm{Edges}(F_d/N)},\nu_{\mathcal{F}})$ is a $\mu$-boundary. Let us now consider the group $B=F_d/N$ and the canonical projection $\widetilde{\pi}:F_d/[N,N]\to B$. Denote by $(\partial B,\nu_B)$ the Poisson boundary of the $\widetilde{\pi}_{*}\mu$-random walk on $B$, and denote by $\mathrm{bnd}_B:B^{\mathbb{N}}\to \partial B$ the corresponding boundary map. Then the map
\begin{align*}
		\left(F_d/[N,N]\right)^{\mathbb{N}}&\to \Z^{\mathrm{Edges}(F_d/N) } \times \partial B\\
		w&\mapsto \left(\mathcal{F}(w), \mathrm{bnd}_B\left( \{\widetilde{\pi}(w_n)\}_{n\ge 0} \right) \right),
\end{align*}
pushes forward $\P$ to a $\mu$-stationary measure $\nu$ on $ \Z^{\mathrm{Edges}(F_d/N)}\times \partial B$. With this, the space $( \Z^{\mathrm{Edges}(F_d/N)}\times \partial B,\nu)$ is a $\mu$-boundary of $F_d/[N,N]$.

\begin{lem}\label{lem: via MAGNUS equal entropies}
	Let $\mu$ be a probability measure on the group $F_d/[N,N]$, and denote by $\widetilde{\phi}:F_d/[N,N]\to \Z^d\wr F_d/N$ the Magnus embedding. Denote by $\{w_n\}_{n\ge 0}$ a sample path for the $\mu$-random walk on $F_d/[N,N]$. Then, denoting $B=F_d/N$, we have
	\begin{equation}\label{eq: general entropy MAGNUS}
			H_{\Z^{\mathrm{Edges}(F_d/N)}\times \partial B}(w_n)=H_{(\Z^d)^B\times \partial B}\left(\widetilde{\phi}(w_n)\right), \text{ for every }n\ge 1, \text{ and}
\end{equation}
	\begin{equation}\label{eq: special entropy MAGNUS}
	H_{\Z^{\mathrm{Edges}(F_d/N)}}(w_n)=H_{(\Z^d)^B }\left(\widetilde{\phi}(w_n)\right), \text{ for every }n\ge 1.
\end{equation}

\end{lem}
\begin{proof}
	This result follows from  Lemma \ref{lem: flows stabilize iff lamps stabilize}, which states that the edgewise limit flow is completely determined by the pointwise limit lamp configuration, as well as the other way around. To prove Equation \eqref{eq: general entropy MAGNUS}, note that we have
	\begin{align*}
			H_{\Z^{\mathrm{Edges}(F_d/N)}\times \partial B}(w_n)&=\int_{\Z^{\mathrm{Edges}(F_d/N)}\times \partial B}H_{(\mathfrak{f}_{\infty},\xi)}(w_n)\ d\nu(\mathfrak{f}_{\infty},\xi)\\
			&=\int_{(\Z^d)^B\times \partial B}H_{(\varphi_{\infty},\xi)}(\widetilde{\phi}(w_n))\ d\nu(\varphi_{\infty},\xi)\\
			&=H_{(\Z^d)^B\times \partial B}\left(\widetilde{\phi}(w_n)\right).
	\end{align*}

Equation \eqref{eq: special entropy MAGNUS} is proved in a completely analogous way, noting that Lemma \ref{lem: flows stabilize iff lamps stabilize} does not depend on what the Poisson boundary of the induced random walk on $B$ is.
\end{proof}
We finish the paper with the proofs of Corollaries \ref{cor: general description for general Magnus embedding} and \ref{cor: description for free solvable group}.	
\begin{proof}[Proof of Corollary \ref{cor: general description for general Magnus embedding}]
 Denote by $\widetilde{\phi}:F_d/[N,N]\to \Z^d\wr F_d/N$ the Magnus embedding. If $\mu$ is a probability measure with finite entropy, then $\widetilde{\phi}_{*}\mu$ also has finite entropy. Furthermore, if the flows stabilize almost surely along sample paths of the $\mu$-random walk on $F_d/[N,N]$, then Lemma \ref{lem: flows stabilize iff lamps stabilize} implies that the lamp configuration stabilizes along sample paths of the $\widetilde{\phi}_{*}\mu$-random walk on $\Z^d\wr F_d/N$. Then, Theorem \ref{thm: lamplighter Poisson boundary} implies that the Poisson boundary of $(\Z^d\wr F_d/N, \widetilde{\phi}_{*}\mu)$ is the space $ (\Z^d)^B\times \partial B $ with the corresponding harmonic measure, where we denote $B\coloneqq F_d/N.$ Then, the conditional entropy criterion (Theorem \ref{thm: conditional entropy alternative formulation}) implies that $\lim_{n\to \infty} \frac{H_{(\Z^d)^B\times \partial B}\left(\widetilde{\phi}(w_n)\right)}{n}=0,$ and we conclude from Equation \eqref{eq: general entropy MAGNUS} of Lemma \ref{lem: via MAGNUS equal entropies}, that
\begin{equation*}
	\lim_{n\to \infty} \frac{H_{\Z^{\mathrm{Edges}(F_d/N)}\times \partial B}(w_n)}{n}=	\lim_{n\to \infty} \frac{H_{(\Z^d)^B\times \partial B}\left(\widetilde{\phi}(w_n)\right)}{n}=0.
\end{equation*}
Hence, $\Z^{\mathrm{Edges}(F_d/N)}\times \partial B$ with the corresponding hitting measure is the Poisson boundary of $(F_d/[N,N],\mu)$.
	
	If we furthermore suppose that $\mu$ has a finite first moment, then $\widetilde{\phi}_{*}\mu$ also has a finite first moment. Then Theorem \ref{thm: Poisson for finite first moment} together with the conditional entropy criterion imply that $\lim_{n\to \infty} \frac{H_{(\Z^d)^B}\left(\widetilde{\phi}(w_n)\right)}{n}=0.$ Finally, we use Equation \eqref{eq: special entropy MAGNUS} from Lemma \ref{lem: via MAGNUS equal entropies} to obtain
\begin{equation*}
	\lim_{n\to \infty} \frac{H_{\Z^{\mathrm{Edges}(F_d/N)}}(w_n)}{n}=	\lim_{n\to \infty} \frac{H_{(\Z^d)^B}\left(\widetilde{\phi}(w_n)\right)}{n}=0,
	\end{equation*}
	so that $\Z^{\mathrm{Edges}(F_d/N)}$ endowed with the corresponding hitting measure is the Poisson boundary of $(F_d/[N,N],\mu)$.
\end{proof}

\begin{proof}[Proof of Corollary \ref{cor: description for free solvable group}]
		Consider the free solvable group $S_{d,k}$ of rank $d\ge 2$ and derived length $k\ge 2$. Let $\mu$ be an adapted probability measure on $S_{d,k}$, and suppose that $\mu$ has a finite first moment. Denote by $\widetilde{\phi}:S_{d,k}\to \Z^d\wr S_{d,k-1}$ the Magnus embedding.
	
	Suppose that $\mu$ induces a transient random walk on $S_{d,k-1}$, via the projection to $\widetilde{\pi}:S_{d,k}\to S_{d,k-1}$. Note that this is always the case if $d\ge 3$ or $k\ge 3$. Indeed, in such cases the group $S_{d,k-1}$ grows at least cubically and hence any adapted measure will induce a transient random walk \cite{Varopoulos1986} (see also \cite[Theorem 3.24]{Woess2000}). Then, it follows from the second item of Corollary \ref{cor: general description for general Magnus embedding} that the Poisson boundary of $(\Z^d\wr S_{d,k-1},\widetilde{\phi}_{*}\mu)$ is $(\Z^d)^{S_{d,k-1}}$, endowed with the corresponding hitting measure.

	 Suppose now that $\mu$ induces a recurrent random walk on $S_{d,k-1}$. As we remarked above, since $\mu$ is adapted, this can only happen when $d=k=2$. In this case, the Poisson boundary of the $\mu$-random walk on  $S_{2,2}$ coincides with the Poisson boundary of the $\widetilde{\phi}_{*}\mu$-random walk on $\Z^2\wr \Z^2$. Since $\widetilde{\phi}_{*}\mu$ induces a recurrent random walk on $\Z^2$, the aforementioned Poisson boundary is isomorphic to the Poisson boundary of the induced random walk on the subgroup $\bigoplus_{\Z^2}\Z^2$ (see \cite[Lemma 4.2]{Furstenberg1971}). Since this subgroup is abelian, we conclude the triviality of the Poisson boundary.
\end{proof}

\bibliographystyle{alpha}
\bibliography{biblio}
\end{document}